\documentclass[12pt]{amsart}
\usepackage{ amsmath,amssymb,amsthm,color, euscript, enumerate,amsfonts,longtable}
\usepackage{comment}
\usepackage{dsfont}
\usepackage{longtable}

\newcommand{\1}{ \mathds{1}}

\newcommand{\Z}{\mathbb{Z}}
\newcommand{\C}{\mathbb{C}}
\newcommand{\R}{\mathbb{R}}

\newcommand{\g}{\mathfrak{g}}
\newcommand{\s}{\mathfrak{s}}
\newcommand{\h}{\mathfrak{h}}
\newcommand{\F}{\mathbb{F}}
\newcommand{\Sym}{\mathrm{Sym}}
\newcommand{\Alt}{\mathrm{Alt}}
\newcommand{\Imm}{\mathrm{Im}}
\newcommand{\Ker}{\mathrm{Ker}}
\newcommand{\End}{\mathrm{End}}
\newcommand{\Aut}{\mathrm{Aut}}
\newcommand{\Inn}{\mathrm{Inn}}

\newcommand{\orb}{\mathrm{orb}}

\newcommand{\Hom}{\mathrm{Hom}}
\newcommand{\Out}{\mathrm{Out}}
\newcommand{\Irr}{\mathrm{Irr}}
\makeatletter \@addtoreset{equation}{section}

\theoremstyle{plain}
\newtheorem{theorem}{Theorem}[section]
\newtheorem{proposition}[theorem]{Proposition}
\newtheorem{lemma}[theorem]{Lemma}
\newtheorem{corollary}[theorem]{Corollary}

\theoremstyle{definition}

\theoremstyle{remark}
\newtheorem{remark}[theorem]{Remark}
\numberwithin{equation}{section}

\title[Automorphism groups of the holomorphic VOAs]{Automorphism groups of the holomorphic vertex operator algebras associated with Niemeier lattices and the $-1$-isometries}
\author[H. Shimakura]{Hiroki Shimakura}%
\address[H. Shimakura]{Graduate School of Information Sciences,
Tohoku University,
Sendai 980-8579, Japan }%
\email {shimakura@tohoku.ac.jp}%
\date{}
\thanks{H.\ Shimakura was partially supported by JSPS KAKENHI Grant Number JP17K05154.}

\begin{document}

\begin{abstract} In this article, we determine the automorphism groups of $14$ holomorphic vertex operator algebras of central charge $24$ obtained by applying the $\Z_2$-orbifold construction to the Niemeier lattice vertex operator algebras and lifts of the $-1$-isometries.

\end{abstract}
\maketitle


\section{Introduction}
Recently, (strongly regular) holomorphic vertex operator algebras (VOAs) of central charge $24$ with non-zero weight one spaces are classified; there exist exactly $70$ such VOAs (up to isomorphism) and they are uniquely determined by the Lie algebra structures on the weight one spaces.
The remaining case is the famous conjecture in \cite{FLM}: a (strongly regular) holomorphic VOA of central charge $24$ is isomorphic to the moonshine VOA if the weight one space is zero. 

The determination of the automorphism groups of vertex operator algebras is one of fundamental problems in VOA theory; it is natural to ask what the automorphism groups of holomorphic VOAs of central charge 24 are.
For example, the automorphism group of the moonshine VOA is the Monster (\cite{FLM}) and those of Niemeier lattice VOAs were determined in \cite{DN}.
However, the other cases have not been determined yet.

The purpose of this article is to determine the automorphism group of the holomorphic VOA $V_N^{\orb({\theta})}$ of central charge $24$ obtained in \cite{DGM} by applying the $\Z_2$-orbifold construction to the lattice VOA $V_N$ associated with a Niemeier lattice $N$ and a lift $\theta$ of the $-1$-isometry of $N$.
Since $V_N^{\orb(\theta)}$ is isomorphic to some Niemeier lattice VOA for $9$ cases and $V_\Lambda^{\orb(\theta)}$ is the moonshine VOA for the Leech lattice $\Lambda$, we focus on the other $14$ cases.
Our main theorem is as follows:

\begin{theorem}\label{T:main} Let $N$ be a Niemeier lattice whose root sublattice is $A_2^{12}$, $A_3^8$, $A_4^6$, $A_5^4D_4$, $A_6^4$, $A_7^2D_5^2$, $A_8^3$, $A_9^2D_6$, $E_6^4$, $A_{11}D_7E_6$, $A_{12}^2$, $A_{15}D_9$, $A_{17}E_7$ or $A_{24}$. 
For the holomorphic VOA $V=V_N^{{\rm orb}(\theta)}$ of central charge $24$, the groups $K(V)$, $\Out_1(V)$ and $\Out_2(V)$ are given as in Table \ref{T:Q}.
\end{theorem}

\begin{small}
\begin{longtable}[c]{|c|c|c|c|c|c|c|c|}
\caption{ $K(V)$, $\Out_1(V)$ and $\Out_2(V)$ for $V={V}_N^{\orb(\theta)}$} \label{T:Q}\\
\hline
No. in \cite{Sc93}&$Q$&$V_1$&${\rm rank}\, V_1$&$K(V)$&$\Out_1(V)$&$\Out_2(V)$&$\S$\\ \hline
$2$&$A_2^{12}$&$A_{1,4}^{12}$&$12$&$\Z_2$&$1$&$M_{12}$&\\ \hline
$5$&$A_3^8$&$A_{1,2}^{16}$&$16$&$\Z_2^5$&$1$&$\Z_2^4:L_4(2)$&\ref{S:A38}\\ \hline
$12$&$A_4^6$&$B_{2,2}^6$&$12$&$\Z_2$&$1$&$\Sym_5$&\\ \hline
$16$&$A_5^4D_4$&$A_{3,2}^4A_{1,1}^4$&$16$&$\Z_4\times\Z_2^3$&$\Z_2$&$\Z_2^4:\Sym_3$&\ref{S:A54D4}\\ \hline
$23$&$A_6^4$&$B_{3,2}^{4}$&$12$&$\Z_2$&$1$&${\rm Alt}_4$&\\ \hline
$25$&$A_7^2D_5^2$&$D_{4,2}^2B_{2,1}^{4}$&$16$&$\Z_2^3$&$1$&$\Sym_2\times\Sym_4$
&\ref{S:A72D52}\\ \hline
$29$&$A_8^3$&$B_{4,2}^{3}$&$12$&$\Z_2$&$1$&$\Sym_3$&\\ \hline
$31$&$A_9^2D_6$&$D_{5,2}^{2}A_{3,1}^{2}$&$16$&$\Z_4^2$&$\Z_2$&$\Sym_2\times\Sym_2$
&\ref{S:A92D6}\\ \hline
$38$&$E_6^4$&$C_{4,1}^{4}$&$16$&$\Z_2$&$1$&$\Sym_4$&\\ \hline
$39$&$A_{11}D_7E_6$&${D_{6,2}}B_{3,1}^{2}C_{4,1}$&$16$&$\Z_2^2$&$1$&$\Sym_2$
&\ref{S:A11D7E6}\\ \hline
$41$&$A_{12}^2$&$B_{6,2}^{2}$&$12$&$\Z_2$&$1$&$\Sym_2$&\\ \hline
$47$&$A_{15}D_9$&${D_{8,2}}B_{4,1}^{2}$&$16$&$\Z_2^2$&$1$&$\Sym_2$&\ref{S:A15D9}\\ \hline
$50$&$A_{17}E_7$&${D_{9,2}}{A_{7,1}}$&$16$&$\Z_8$&$\Z_2$&$1$&\ref{S:A17E7}\\ \hline
$57$&$A_{24}$&${B_{12,2}}$&$12$&$\Z_2$&$1$&$1$&\\ \hline
\end{longtable}
\end{small}

The groups $K(V)$, $\Out_1(V)$ and $\Out_2(V)$ in the theorem above are defined as follows (see also Section 2.3).
Let $V$ be a holomorphic VOA of central charge $24$ such that $V_1$ is semisimple.
Since the Lie algebra structure of $V_1$ determines the VOA structure of $V$, we focus on the action of the automorphism group $\Aut(V)$ on $V_1$; let $K(V)$ be the subgroup of $\Aut(V)$ which acts trivially on $V_1$.
Then $\Aut(V)/K(V)$ is a subgroup of $\Aut(V_1)$.
Let $\Inn(V)$ be the inner automorphism group of $V$.
Then $\Inn(V)/(K(V)\cap\Inn(V))$ is isomorphic to the inner automorphism group $\Inn(V_1)$ of the Lie algebra $V_1$.
Let $\Out(V)$ be the quotient group $\Aut(V)/(K(V)\Inn(V))$, which is a subgroup of the outer automorphism group $\Out(V_1)=\Aut(V_1)/\Inn(V_1)$ of $V_1$.
Let $\Out_1(V)$ be the subgroup of $\Out(V)$ which preserves every simple ideal of $V_1$ and set $\Out_2(V)=\Out(V)/\Out_1(V)$.
Then $\Aut(V)$ is described by $K(V)$, $\Out_1(V)$, $\Out_2(V)$ and $\Inn(V_1)$.

\medskip

Let us explain how to determine the automorphism group of the holomorphic VOA $V=V_N^{\orb(\theta)}=V_N^+\oplus V_N(\theta)_\Z$.
Since the conformal weight of $V_N(\theta)_\Z$ is two, we have $V_1=(V_N^+)_1$.
Let $Q=\bigoplus_{i=1}^t Q_i$ be the decomposition of the root sublattice $Q$ of $N$ into the orthogonal sum of indecomposable root lattices $Q_i$.
Then $V_1\cong \bigoplus_{i=1}^t(V_{Q_i}^+)_1$.

First, we recall (known) automorphisms of $V$.
Clearly the map $z$ which acts on $V_N^+$ and $V_N(\theta)_\Z$ by $1$ and $-1$, respectively, is an order $2$ automorphism of $V$.
Since $V$ is a $\Z_2$-graded simple current extension of $V_N^+$ and $N$ is unimodular, the centralizer $C_{\Aut(V)}(z)$ in $\Aut(V)$ of $z$ is a central extension of $\Aut(V_N^+)$ by $\langle z\rangle$ (\cite{Sh04,Sh06}).
Note that the group $\Aut(V_N^+)$ is well-studied in \cite{Sh04,Sh06}.
In addition, if $N$ is constructed from a binary code by the same manner as the Leech lattice, $V$ has an extra automorphism not in $C_{\Aut(V)}(z)$ (\cite{FLM}).
In fact, they generate $\Aut(V)$ (Corollary \ref{C:AutV}).

Next, we determine the group $K(V)$.
Since $K(V)$ preserves $V_N^+$, we have $K(V)\subset C_{\Aut(V)}(z)$, and hence $K(V)/\langle z\rangle\subset C_{\Aut(V)}(z)/\langle z\rangle\cong\Aut(V_N^+)$.
By the definition of $K(V)$, $K(V)/\langle z\rangle$ acts trivially on $V_1=(V_N^+)_1$.
Hence, the lift of $K(V)/\langle z\rangle$ in $\Aut(V_N)$ preserves the  Cartan subalgebra $\C\otimes_\Z N$ of $(V_N)_1$, and $K(V)/\langle z\rangle\subset O(\hat{N})/\langle\theta\rangle$, where $O(\hat{N})$ is the lift of the isometry group of $N$ in $\Aut(V_N)$.
Then we can describe $K(V)/\langle z\rangle$ by using the explicit action of $O(\hat{N})/\langle z\rangle$ on $(V_N^+)_1$.
On the other hand, $K(V)$ contains inner automorphisms $\sigma_x=\exp(-2\pi\sqrt{-1}x_{(0)})$ associated with vectors $x$ in the coweight lattice $P^\vee$ of the Lie algebra $V_1$.
By using the isomorphism between the VOA $V_{Q_i}^+$ and the simple affine VOA associated with the Lie algebra $(V_{Q_i}^+)_1$ at some positive integral level (Proposition \ref{P:VOAgen1}), we prove that $K(V)$ coincides with $\{\sigma_x\mid x\in P^\vee\}$ and determine the group structure of $K(V)$.
In particular, $K(V)$ is a subgroup of $\Inn(V)$.

Finally, we determine the groups $\Out_1(V)$ and $\Out_2(V)$.
If the semisimple Lie algebra $V_1$ has no diagram automorphisms, then $\Out_1(V)=1$.
If all ideals $(V_{Q_i}^+)_1$ of $V_1$ are simple, then $\Out_2(V)$ is obtained from the automorphism group of the glue code $N/Q$, which is described in \cite{CS}.
Then $7$ cases are done.
For the remaining $7$ cases, we consider the set $C_N$ consisting of (isomorphism classes of) simple current $\langle V_1\rangle$-submodules of $V_{N_0}^+$, where $N_0=N\cap (Q/2)$; this set $C_N$, called ``Glue" in \cite{Sc93}, has a group structure under the fusion product.
We prove that $\Aut(V)$ preserves $V_{N_0}^+$, which shows that  $\Out(V)\subset\Aut(C_N)$.
By the description of the glue code $N/Q$ in \cite{CS}, we obtain a generator of $C_N$ and determine $\Aut(C_N)$.
Considering the explicit actions of (known) automorphisms of $V$ on $C_N$, we prove that $\Out(V)=\Aut(C_N)$ and determine the group structure of $\Out_1(V)$ and $\Out_2(V)$.

\medskip

The organization of this article is as follows: In Section 2, we review basic facts about integral lattices and VOAs.
In Section 3, we briefly review lattice VOAs $V_L$, subVOAs $V_L^+$ and simple affine VOAs.
In Section 4, we prove that the VOA $V_R^+$ associated with indecomposable root lattice $R$ is a simple affine VOA at some positive integral level if $R\not\cong A_1$.
We also study automorphisms and irreducible modules for $V_R^+$ via the isomorphism.
In Section 5, we determine $\Aut(V_L^{\orb(\theta)})$ under some assumptions on even unimodular lattices $L$.
In Section 6, we prove Theorem \ref{T:main}, the main theorem of this article.

\medskip

\paragraph{\bf Acknowledgments.} Part of this work was done while the author was staying at Institute
of Mathematics, Academia Sinica, Taiwan in August, 2018. He is grateful to the institute.
He also would like to thank Ching Hung Lam for helpful comments.

\begin{center}
{\bf Notations}
\begin{footnotesize}
\begin{tabular}{ll}\\
$(\cdot|\cdot)$& the positive-definite symmetric bilinear form on $\R^r$ or\\ 
& the normalized Killing form so that $(\alpha|\alpha)=2$ for long roots $\alpha$.\\
$A:B$& a split extension of $B$ by $A$, which has the normal subgroup $A$ and a complement $B$.\\
$C_G(a)$& the centralizer of $a$ in a group $G$.\\
$C_N$& the subgroup of $S_N$ consisting of all simple current $\langle (V_N^+)_1\rangle$-submodules of $V_N^+$.\\
$f_u$& the element of $\Hom(L,\Z_2)$ associated with $u\in L^*$, which sends $v\in L$ to $(-1)^{(u|v)}$.\\
$G_2(N)$& the permutation group induced from the action of $O(N)$ on the glue code $N/Q$.\\
$\theta$& an element in $O(\hat{L})$ whose image is $-1$ under the canonical map $O(\hat{L})\to O(L)$.\\
$\Inn(V)$& the normal subgroup of $\Aut(V)$ generated by inner automorphisms of a VOA $V$.\\
$\Irr(V)$& the set of (isomorphism classes of) irreducible modules for a VOA $V$.\\
$K(V)$& the subgroup of $\Aut(V)$ which acts trivially on $V_1$.\\
$L^*$& the dual lattice of a lattice $L$.\\
$L_\mathfrak{g}(k,0)$& the simple affine VOA associated with the semisimple Lie algebra $\mathfrak{g}$ at level $k$.\\
$L_\g(k,\Lambda)(=[\Lambda])$& the irreducible $L_\g(k,0)$-module with the highest weight $\Lambda$.\\
$\lambda_i$& the fundamental weights of a root lattice (with respect to fixed simple roots).\\
$\Lambda_i$& the fundamental weights of a simple Lie algebra (with respect to fixed simple roots).\\
$(\mu)^\pm$, $(\mu)$, $(\chi)^\pm$& the (isomorphism classes of) irreducible $V_L^+$-modules.\\
$M\circ g$& the module $(M,Y_{M\circ g})$, where $Y_{M\circ g}(v,z)=Y_M(gv,z)$.\\
$O(L)$& the isometry group of a lattice $L$.\\
$O(\hat{L})$& the lift of $O(L)$ to a subgroup of $\Aut(\hat{L})$. \\
$\Out(V)$& the image of the group homomorphism $\Aut(V)\to\Out(V_1)$ in Section 2.3 for a VOA $V$.\\
$\Out_1(V)$& the subgroup of $\Out(V)$ which preserves every simple ideal of $V_1$.\\
$\Out_2(V)$& the quotient group $\Out(V)/\Out_1(V)$.\\
$S_N$& the abelian group consisting of all simple current modules for the VOA $\langle (V_N^+)_1\rangle$.\\ 
$\sigma_v$& the inner automorphism $\exp(-2\pi\sqrt{-1}v_{(0)})$ of a VOA $V$ associated with $v\in V_1$.\\
$\Sym_n$& the symmetric group of degree $n$.\\
$V_L$& the lattice VOA associated with the even lattice $L$.\\
$V_L^{+}$& the fixed-point subspace of $\theta$, which is a full subVOA of $V_L$.\\
$V_L^{\orb(\theta)}$& the holomoprhic VOA obtained by applying the $\Z_2$-orbifold construction to $V_L$ and $\theta$.\\
$X_n(=A_n,\dots)$& (the type of) a root system, a simple Lie algebra or a root lattice.\\
$X_{n,k}$& (the type of) a simple Lie algebra whose type is $X_n$ and level is $k$.\\
$\Z_n$&the cyclic group of order $n$.\\
\end{tabular}
\end{footnotesize}
\end{center}

\section{Preliminary}
In this section, we review basics about integral lattices and VOAs.

\subsection{Lattices}
Let $(\cdot|\cdot)$ be a positive-definite symmetric bilinear form on $\R^r$.
A subset $L$ of $\R^r$ is called a (positive-definite) \emph{lattice} of rank $r$ if $L$ has a basis $e_1,e_2,\dots,e_r$ of $\R^r$ satisfying $L=\bigoplus_{i=1}^r\Z e_i$.
Let $L^*$ denote the dual lattice of a lattice $L$ of rank $r$, that is, $L^*=\{v\in \R^r\mid ( v| L)\subset\Z\}$.
For $v\in \R^m$, we call $(v|v)$ the (squared) \emph{norm} of $v$.
A lattice $L$ is said to be \emph{even} if the norm of any vector in $L$ is even, and is said to be \emph{unimodular} if $L=L^*$.
A lattice is (orthogonally) \emph{indecomposable} if it cannot be written as an orthogonal sum of proper sublattices.
It is known that any lattice can be uniquely written as the orthogonal sum of indecomposable sublattices.
A group automorphism $g$ of a lattice $L$ is call an \emph{isometry} of $L$ if $(g(v)|g(w))=(v|w)$ for all $v,w\in L$; 
let $O(L)$ denote the isometry group of $L$.

Let $L$ be an even lattice.
Then the set of norm $2$ vectors of $L$ forms a root system.
Let $Q$ be the sublattice of $L$ generated by norm $2$ vectors of $L$; this sublattice is often called the \emph{root sublattice} of $L$.
We call $L$ a \emph{root lattice} if $L=Q$.
It is known (e.g. \cite{Hu}) that the root system of an indecomposable root lattice is of type $A_\ell$ ($\ell\ge1$), $D_m$ ($m\ge4$) or $E_n$ $(n=6,7,8)$.
We often denote the root lattice by the type of its root system.

Even unimodular lattices of rank $24$, called \emph{Niemeier lattices}, are classified by Niemeier as follows:
\begin{theorem}[\cite{Nie}] There exist Niemeier lattices with the following $24$ root sublattices:\\
\begin{tabular}{llllllll}
$A_{1}^{24}$, & $A_{2}^{12}$, &  $A_{3}^{8}$, &  $A_{4}^6$, & $A_{5}^4D_{4}$, & $D_{4}^6$, &  $A_{6}^4$, & $A_{7}^2D_{5}^2$,\\ 
$A_{8}^3$, & $A_{9}^2D_{6}$, & $D_{6}^4$, & $E_{6}^4$, & $A_{11}D_{7}E_{6}$, & $A_{12}^2$, & $D_{8}^3$, & $A_{15}D_{9}$,\\  $A_{17}E_{7}$, & $D_{10}E_{7}^2$, & $D_{12}^2$, & $A_{24}$, & $D_{16}E_{8}$, & $E_{8}^3$, & $D_{24}$, & $0$.
\end{tabular}\\
Moreover, the isometry class of a Niemeier lattice is uniquely determined by that of its root sublattice.
In particular, there exist exactly $24$ Niemeier lattices, up to isometry.
\end{theorem}

Let us recall the automorphism groups of Niemeier lattices from \cite[Chapter 16]{CS}.
Let $N$ be a Niemeier lattice with the root sublattice $Q\neq0$.
Let $W(Q)$ be the Weyl group, the normal subgroup of $O(N)$ generated by the reflections associated with roots in $Q$.
Let $Q=\bigoplus_{i=1}^tQ_i$ be the orthogonal sum of indecomposable root lattices and let $H(Q)$ be the subgroup of $O(Q)$ generated by diagram automorphisms of $Q_i$ and possible permutations on $\{Q_i\mid 1\le i\le t\}$.
Set $H(N)=H(Q)\cap O(N)$.
Then $O(N)$ is a split extension of $H(N)$ by $W(Q)$.
Let $G_1(N)$ be the subgroup of $H(N)$ that preserves every indecomposable component $Q_i$ and set $G_2(N)=H(N)/G_1(N)$.
Then $G_2(N)$ acts on $\{Q_i\mid 1\le i\le t\}$ as a permutation group.

The following lemma follows immediately from the fact that the Weyl group of an indecomposable root lattice contains the $-1$-isometry if and only if the root lattice is $A_1$, $D_{2n}$ ($n\ge2$), $E_7$ or $E_8$.

\begin{lemma}\label{L:-1G0} Assume that $Q\neq0$.
Then the $-1$-isometry of $N$ does not belong to $W(Q)$ if and only if $Q$ is $A_2^{12}$, $A_3^8$, $A_4^6$, $A_5^4D_4$, $A_6^4$, $A_7^2D_5^2$, $A_8^3$, $A_9^2D_6$, $E_6^4$, $A_{11}D_7E_6$, $A_{12}^2$, $A_{15}D_9$, $A_{17}E_7$ or $A_{24}$.
\end{lemma}

\begin{remark}\label{L:-1} Assume that $Q$ is one in Lemma \ref{L:-1G0}, that is, $-1\notin W(Q)$.
Then $G_1(N)\cong\Z_2$ (\cite[Table 16.1]{CS}).
Hence $\{g\in O(N)\mid g=\pm1\ {\rm on}\ Q_i,\ 1\le \forall i\le t\}/\langle-1\rangle$ is isomorphic to $\Z_2$ if $Q$ is $A_5^4D_4$, $A_9^2D_6$ or $A_{17}E_7$; it is equal to $1$ otherwise.
\end{remark}

We summarize the groups $G_2(N)$ for the $14$ Niemeier lattices in Lemma \ref{L:-1G0} from \cite[Chapter 16]{CS}.

\begin{table}[bht]
\caption{Groups $G_2(N)$} \label{T:G2N}
\begin{tabular}{|c|c|c|c|c|c|c|c|}
\hline
$Q$& $A_2^{12}$& $A_3^8$& $A_4^6$& $A_5^4D_4$& $A_6^4$& $A_7^2D_5^2$& $A_8^3$ \\\hline
$G_2(N)$&$M_{12}$&$\Z_2^3:L_3(2)$&$\Sym_5$&$\Sym_4$&$\Alt_4$&$\Sym_2^2$&$\Sym_3$
\\
 \hline\hline
$Q$& $A_9^2D_6$& $E_6^4$& $A_{11}D_7E_6$& $A_{12}^2$& $A_{15}D_9$& $A_{17}E_7$ & $A_{24}$\\\hline
$G_2(N)$&$\Sym_2$&$\Sym_4$&$1$&$\Sym_2$&$1$&$1$&$1$\\ \hline
\end{tabular}
\end{table}

\begin{remark} In  Table \ref{T:G2N}, we denote by $A:B$ a split extension of a group $B$ by a group $A$.
We also denote by $L_n(2)$ and $M_{12}$ the (projective special) linear group on $\F_2^n$ and the Mathieu group of degree $12$, respectively. 
\end{remark}

\subsection{Vertex operator algebras}
A \emph{vertex operator algebra} (VOA) $(V,Y,\1,\omega)$ is a $\Z$-graded vector space $V=\bigoplus_{m\in\Z}V_m$ over the complex field $\C$ equipped with a linear map
$$Y(a,z)=\sum_{i\in\Z}a_{(i)}z^{-i-1}\in ({\rm End}\ V)[[z,z^{-1}]],\quad a\in V,$$
the \emph{vacuum vector} $\1\in V_0$ and the \emph{conformal vector} $\omega\in V_2$
satisfying certain axioms (\cite{Bo,FLM}). 
The operators $L(m)=\omega_{(m+1)}$, $m\in \Z$, satisfy the Virasoro relation:
$$[L{(m)},L{(n)}]=(m-n)L{(m+n)}+\frac{1}{12}(m^3-m)\delta_{m+n,0}c\ {\rm id}_V,$$
where $c\in\C$ is called the \emph{central charge} of $V$.
A \emph{vertex operator subalgebra} (or a \emph{subVOA}) is a graded subspace of
$V$ which has a structure of a VOA such that the operations and its grading
agree with the restriction of those of $V$ and  they share the vacuum vector.
A subVOA is said to be \emph{full} if it has the same conformal vector as $V$.

For a VOA $V$, a $V$-\emph{module} $(M,Y_M)$ is a $\C$-graded vector space $M=\bigoplus_{m\in\C} M_{m}$ equipped with a linear map
$$Y_M(a,z)=\sum_{i\in\Z}a_{(i)}z^{-i-1}\in (\End\ M)[[z,z^{-1}]],\quad a\in V$$
satisfying a number of conditions (\cite{FHL,DLM2}).
We often denote it by $M$.
If $M$ is irreducible, then there exists $w\in\C$ such that $M=\bigoplus_{m\in\Z_{\geq 0}}M_{w+m}$ and $M_w\neq0$; the number $w$ is called the \emph{conformal weight} of $M$.
Let $\Irr(V)$ denote the set of all isomorphism classes of irreducible $V$-modules.
We often identify an element in $\Irr(V)$ with its representative.

A VOA is said to be  \emph{rational} if its admissible module category is semisimple.
(See \cite{DLM2} for the definition of admissible modules.)
A rational VOA is said to be \emph{holomorphic} if it itself is the only irreducible module up to isomorphism.
A VOA $V$ is said to be \emph{of CFT-type} if $V_0=\C\1$ (note that $V_n=0$ for all $n<0$ if $V_0=\C\1$), and is said to be \emph{$C_2$-cofinite} if the codimension in $V$ of the subspace spanned by the vectors of form $u_{(-2)}v$, $u,v\in V$, is finite.
A module is said to be \emph{self-dual} if it is isomorphic to its contragredient module (\cite{FHL}).
A VOA is said to be \emph{strongly regular} if it is rational, $C_2$-cofinite, self-dual and of CFT-type.
Note that a strongly regular VOA is simple.

Let $V$ be a VOA of CFT-type.
Then, the weight one space $V_1$ has a Lie algebra structure via the $0$-th product.
Moreover, the operators $v_{(m)}$, $v\in V_1$, $m\in\Z$, define  an affine representation of the Lie algebra $V_1$ on $V$.
For a simple Lie subalgebra $\mathfrak{s}$ of $V_1$, the \emph{level} of $\mathfrak{s}$ is defined to be the scalar by which the canonical central element acts on $V$ (as the affine representation);
when the type of $\mathfrak{s}$ is $X_n$ and the level of $\mathfrak{s}$ is $k$, we denote the type of $\mathfrak{s}$ (with level) by $X_{n,k}$.

\begin{proposition}[{\cite[Theorem 1.1, Corollary 4.3]{DM06}}]\label{Prop:posl} Let $V$ be a strongly regular VOA.
Then $V_1$ is reductive.
Let $\mathfrak{s}$ be a simple Lie subalgebra of $V_1$.
Then $V$ is an integrable module for the affine representation of $\mathfrak{s}$, and the subVOA generated by $\mathfrak{s}$ is isomorphic to the simple affine VOA associated with $\mathfrak{s}$ at some positive integral level.
\end{proposition}

Assume that $V$ is strongly regular.
Then the fusion products are defined on irreducible $V$-modules (\cite{HL}).
An irreducible $V$-module $M^1$ is called a \emph{simple current module} if for any irreducible $V$-module $M^2$, the fusion product $M^1\boxtimes M^2$ is also an irreducible $V$-module.

\subsection{Automorphisms of vertex operator algebras}
Let $V$ be a VOA.
A linear automorphism $g$ of $V$ is called a (VOA) \emph{automorphism} of $V$ if $$ g\omega=\omega\quad {\rm and}\quad gY(v,z)=Y(gv,z)g\quad \text{ for all } v\in V.$$
Let us denote the group of all automorphisms of $V$ by $\Aut(V)$. 
Note that $\Aut(V)$ preserves $V_n$ for every $n\in\Z$.

Assume that $V$ is of CFT-type.
Then for $v\in V_1$, $\exp(v_{(0)})$ is an automorphism of $V_1$, which is called an \emph{inner automorphism}.
Let $\Inn(V)$ denote the normal subgroup generated by inner automorphisms of $V$.
When $v_{(0)}$ is semisimple on $V$, we set $$\sigma_{v}=\exp(-2\pi\sqrt{-1}v_{(0)})\in \Inn(V).$$ 

We further assume that the Lie algebra $V_1$ is semisimple; $V_1=\bigoplus_{i=1}^s \g_i$, where $\g_i$ are simple ideals.
Let $$\varphi_1:\Aut (V)\to \Aut (V_1)$$ be the restriction map and let $$\varphi_2:\Aut(V_1)\to\Out(V_1)=\Aut(V_1)/\Inn(V_1)$$ be the canonical map, where $\Inn(V_1)$ and $\Out(V_1)$ are the inner and the outer automorphism groups of the Lie algebra $V_1$, respectively.
Set 
\begin{equation}
K(V):=\Ker\, \varphi_1\subset\Aut(V),\quad \Out(V):=\Imm (\varphi_2\circ\varphi_1)\subset\Out(V_1).\label{Def:KO}
\end{equation}
Note that $\varphi_1(\Inn(V))=\Inn(V_1)$.
Set
\begin{equation}
\Out_1(V):=\{g\in \Out(V)\mid g(\g_i)=\g_i,\ 1\le \forall i\le s\},\ \Out_2(V):=\Out(V)/\Out_1(V).\label{Def:O12}
\end{equation} 
Then $\Out_2(V)$ acts faithfully on $\{\g_i\mid 1\le i\le s\}$ as a permutation group.

\subsection{Action of automorphisms on modules}

Let $V$ be a VOA and $g\in\Aut(V)$.
Let $M=(M,Y_M)$ be a $V$-module.
The $V$-module $M\circ g$ is defined as follows:
\[
\begin{split}
& M\circ g =M \quad \text{ as a vector space;}\\
& Y_{M\circ g} (a, z) = Y_M(ga, z)\quad \text{ for any } a\in V.
\end{split}
\]

The following lemma is immediate.
\begin{lemma}\label{L:gcon} Let $M,M^1,M^2$ be $V$-modules and let $g\in\Aut(V)$.
\begin{enumerate}[{\rm (1)}]
\item If $M$ is irreducible, then so is $M\circ g$.
\item If $M^1$ and $M^2$ are isomorphic, then so are $M^1\circ g$ and $M^2\circ g$.
\item If $M$ is a simple current module, then so is $M\circ g$.
\end{enumerate}
\end{lemma}

By the lemma above, $g\in\Aut(V)$ acts on $\Irr(V)$ by $W\mapsto W\circ g$.
Note that this action preserves the fusion products and that if $g\in \Inn(V)$ then $W\circ g=W$ for all $W\in \Irr(V)$.

\begin{lemma}\label{L:gact} Let $V$ be a VOA and $U$ a rational full subVOA.
Let $V=\bigoplus_{W\in \Irr(U)}V(W)$ be the isotypic decomposition, where $V(W)$ is the sum of all irreducible $U$-submodules of $V$ that belongs to $W$.
Let $g\in \Aut(V)$ such that $g(U)=U$.
Then for $W\in \Irr(U)$, we have $g(V(W))=V(W\circ g^{-1})$.
\end{lemma}
\begin{proof} Let $M$ be an irreducible $U$-submodule of $V$.
Then $$gY_M(v,z)w=Y_{g(M)}(gv,z)g(w)=Y_{g(M)\circ g}(v,z)g(w),$$
and $g$ is a $U$-module isomorphism from $M$ to $g(M)\circ g$.
Hence $M\circ g^{-1}$ is isomorphic to $g(M)$ as $U$-modules, and $g(V(W))\subset V(W\circ g^{-1})$ for $W\in\Irr(U)$. Replacing $W$ by $W\circ g$ and $g$ by $g^{-1}$, we obtain the opposite inclusion.
\end{proof}

\section{Lattice VOAs $V_L$, subVOAs $V_L^+$ and simple affine VOAs}

In this section, we review properties of lattice VOAs $V_L$, subVOAs $V_L^+$ and simple affine VOAs.

\subsection{Lattice VOAs}
Let $L$ be an even lattice of rank $r$ and let $(\cdot |\cdot )$ be a positive-definite symmetric bilinear form on $\R\otimes_\Z L\cong\R^r$.
The lattice VOA $V_L$ associated with
$L$ is defined to be $M(1) \otimes_\C \C\{L\}$ (\cite{FLM}). 
Here $M(1)$ is the Heisenberg VOA associated with $\mathfrak{h}=\C\otimes_\Z L$ and the form $(\cdot|\cdot)$ extended $\C$-bilinearly, and $\C\{L\}=\bigoplus_{\alpha\in L}\C e^\alpha$ is the twisted group algebra with the commutator relation $e^\alpha e^\beta=(-1)^{(\alpha|\beta)}e^{\beta}e^{\alpha}$ ($\alpha,\beta\in L$).
It is well-known that the lattice VOA $V_L$ is strongly regular, and its central charge is equal to $r$, the rank of $L$.

Let $\hat{L}$ be the central extension of $L$ by $\langle-1\rangle\cong\Z_2$ associated with the commutator relation above.
Let $\Aut(\hat{L})$ be the set of all group automorphisms of $\hat L$.
For $g\in \Aut (\hat{L})$, we define the element $\bar{g}\in\Aut (L)$ by $g(e^\alpha)\in\{\pm e^{\bar{g}(\alpha)}\}$, $\alpha\in L$.
Set $$O(\hat{L})=\{g\in\Aut(\hat L)\mid \bar{g}\in O(L)\}.$$
We often identify $\Hom(L,\Z_2)$ with $\{f_u\mid u\in L^*/2L^*\}$, where 
\begin{equation}
f_u:L\to \{\pm1\}\cong\Z_2,\quad f_u(\alpha)=(-1)^{(u|\alpha)}.\label{Eq:hom}
\end{equation}
Note that $\Hom(L,\Z_2)\cong \Z_2^r$.
For $f_u\in\Hom(L,\Z_2)$, the map $\hat{L}\to\hat{L}$, $e^{\alpha}\mapsto f_u(\alpha)e^{\alpha}$ is an element of $O(\hat{L})$; we view $\Hom(L,\Z_2)$ as a subgroup of $O(\hat{L})$.
It was proved in \cite[Proposition 5.4.1]{FLM} that the following sequence is exact:
\begin{equation}
1 \longrightarrow \mathrm{Hom}(L, \Z_2) { \longrightarrow}
O(\hat{L}) \bar\longrightarrow O(L)\longrightarrow  1.\label{Exact1}
\end{equation}
We view $O(\hat{L})$ as a subgroup of $\Aut (V_L)$ by the following way: for $g\in O(\hat{L})$, 
$$\alpha_1(-n_1)\dots\alpha_m(-n_m)\otimes e^\beta\mapsto \bar{g}(\alpha_1)(-n_1)\dots\bar{g}(\alpha_m)(-n_m)\otimes g(e^\beta)$$
is an automorphism of $V_L$, 
where $n_1,\dots,n_m\in\Z_{>0}$ and $\alpha_1,\dots,\alpha_m,\beta\in L$.
We often identify $\mathfrak{h}$ with $\mathfrak{h}(-1)\1$ via $h\mapsto h(-1)\1$.
Then $f_u=\sigma_{u/2}$ for $u\in L^*$; for the definition of $\sigma_{u/2}$, see Section 2.3.
Hence $\Hom(L,\Z_2)\subset\Inn(V_L)$.
It was proved in \cite[Theorem 2.1]{DN} that $\Aut (V_L)$ is generated by the normal subgroup $\Inn(V_L)$ and the subgroup $O(\hat L)$.

\subsection{VOAs $V_L^+$}
Let $V_L$ be the lattice VOA associated with an even lattice $L$.
Let $\theta$ be an element in $O(\hat{L})$ such that $\bar{\theta}=-1\in O(L)$.
Note that the order of $\theta$ is $2$ and that $\theta$ is unique up to conjugation by elements in $\Aut(V_L)$ (\cite[Appendix D]{DGH}).
Let $V_L^+$ denote the fixed-point subspace of $\theta$.
Then $V_L^+$ is a full subVOA and it is strongly regular (\cite{ABD,DJL12}).
The irreducible $V_L^+$-modules were classified in \cite{AD} as follows:

\begin{theorem}[{\cite[Theorem 7.7]{AD}}]\label{T:irrVQ+} 
Any irreducible $V_L^+$-module is isomorphic to $V_{\mu+L}^\pm$ $(\mu\in L^*\cap(L/2))$, $V_{\nu+L}$ $(\nu\in L^*\setminus (L/2))$ or $V_L^{T_\chi,\pm}$,
where $T_\chi$ is an irreducible module for the group $\hat{L}/\{a^{-1}\theta(a)\mid a\in\hat{L}\}$ with central character $\chi$.
\end{theorem}
For simplicity, we use the following notations for elements of $\Irr(V_L^+)$:
\begin{align}
V_{\mu+L}^\pm\in (\mu)^\pm\ (\mu\in L^*\cap(L/2)),\quad V_{\nu+L}\in (\nu)\ (\nu\in L^*\setminus (L/2)),\quad
V_L^{T_\chi,\pm}\in (\chi)^\pm.\label{Eq:Not}
\end{align}
The irreducible $V_L^+$-modules $V_{\mu+L}^\pm$ and $V_{\nu+L}$ (resp. $V_L^{T_\chi,\pm}$) are said to be of \emph{untwisted type} (resp. of \emph{twisted type}).
Note that $(\nu)=(-\nu)$ for $\nu\in L^*\setminus (L/2)$.
The following lemma is straightforward.

\begin{lemma}\label{L:decuntw} Let $L$ be an even lattice and $Q$ its sublattice.
Assume that $L$ and $Q$ have the same rank.
Let $Q=\bigoplus_{i=1}^tQ_i$ be the orthogonal sum of indecomposable sublattices.
Then for any irreducible $V_L^+$-module of untwisted (resp. twisted) type, any irreducible $\bigotimes_{i=1}^t V_{Q_i}^+$-submodule is isomorphic to the tensor product of irreducible $V_{Q_i}^+$-modules of untwisted (resp. twisted) type.
\end{lemma}

By the fusion rules of $\Irr(V_L^+)$ determined in \cite{ADL}, we obtain the following:

\begin{proposition}\label{P:ADL} Let $M$ be an irreducible $V_L^+$-module of untwisted type.
Then $M$ is a simple current module if and only if $M\in (\mu)^\varepsilon$ for some $\mu\in L^*\cap (L/2)$ and $\varepsilon\in\{\pm\}$.
\end{proposition}

Since $\theta$ is a central element of $O(\hat{L})$ and $\{g\in O(\hat{L})\mid g=id\ {\rm on}\ V_L^+\}=\langle\theta\rangle\cong\Z_2$, we have $O(\hat{L})/\langle\theta\rangle\subset \Aut(V_L^+)$.
By \eqref{Exact1}, we obtain the following exact sequence:
\begin{equation}
1 \longrightarrow \mathrm{Hom}(L, \Z_2) { \longrightarrow}
O(\hat{L})/\langle\theta\rangle \bar\longrightarrow O(L)/\langle-1\rangle\longrightarrow  1.\label{Exact2}
\end{equation}
The actions of $O(\hat{L})/\langle\theta\rangle$ on the subsets $\{(\mu)^\pm\mid \mu\in L^*\cap (L/2)\}$ and $\{(\nu)\mid \nu\in L^*\setminus(L/2)\}$ of $\Irr(V_L^+)$ are described in \cite[Lemma 1.7]{Sh06} (cf.\ \cite[Proposition 2.9]{Sh04}).
In particular, we have the following:

\begin{lemma}\label{L:Hommod} For $u\in L^*$ and $\mu\in L^*\cap (L/2)$, $$(\mu)^\pm\circ f_u=\begin{cases}(\mu)^\pm& {\rm if}\ (\mu|u)\in \Z\\ (\mu)^\mp,& {\rm if}\ (\mu|u)\in \frac{1}{2}+\Z.\end{cases}$$
\end{lemma}

Now, we consider $\Aut(V_Q^+)$ for a root lattice $Q$.

\begin{lemma}\label{L:DE} For a root lattice $Q$, the following are equivalent;
\begin{enumerate}[{\rm (1)}]
\item there exist $W\in \Irr(V_Q^+)$ of untwisted type and $g\in\Aut(V_Q^+)$ such that $W\circ g$ is of twisted type;
\item $Q\cong A_7$, $D_8$, $E_8$ or $D_n\oplus D_n$ $(n\ge2)$, where $D_2=A_1\oplus A_1$ and $D_3=A_3$.
\end{enumerate}
\end{lemma}
\begin{proof} Let $Q$ be a root lattice.
We note that (1) holds if $Q\cong E_8$ since $V_{E_8}^+\cong V_{D_8}$ (cf.\ \cite[Section 3.1]{Sh06}); we assume that $Q\not\cong E_8$.

The assertion (1) holds if and only if $Q$ is obtained by Construction B from a doubly even binary code $C$ (\cite[Proposition 3.10]{Sh04}, \cite[Lemma 1.11, Corollary 4.4]{Sh06}).
Since $Q$ is generated by norm $2$ vectors, $C$ is also generated by weight $4$ cordwords.
Such a doubly even binary code is equivalent to $d_{2n}$ $(n\ge2)$, the Hamming code of length $7$ or the extended Hamming code of length $8$ (\cite[Theorem 6.5]{PS75});
in fact, these binary codes provide the root lattices $D_n\oplus D_n$ $(n\ge2)$, $A_7$ and $D_8$ by Construction B, respectively.
\end{proof}

\begin{remark}\label{R:DE} If $Q$ is $A_7$, $D_8$ or $D_n\oplus D_n$ $(n\ge2)$, then $V_Q^+$ has an extra automorphism of order $2$ as in \cite[Chapter 11]{FLM} that sends some irreducible $V_Q^+$-module of untwisted type to one of twisted type (\cite[Proposition 4.2]{Sh06}).
\end{remark}

The case where $L$ is unimodular was studied in \cite{Sh06}.

\begin{lemma}[{\cite[Proposition 7.3]{Sh06}}]\label{Lem:AutVL+}
Let $L$ be an even unimodular lattice whose rank is at least $16$.
Then $\Aut(V_L^+)\cong C_{\Aut(V_L)}(\theta)/\langle\theta\rangle\cong (C_{\Inn(V_L)}(\theta)O(\hat{L}))/\langle \theta\rangle$.
\end{lemma}

\subsection{Simple affine VOAs associated with simple Lie algebras}

Let $\g$ be a simple Lie algebra.
Let $\mathfrak{h}$ be a (fixed) Cartan subalgebra of $\g$ and let $(\cdot|\cdot)$ be the Killing form on $\g$.
We identify the dual $\mathfrak{h}^*$ with $\mathfrak{h}$ via $(\cdot|\cdot)$ and normalize the form so that $(\alpha|\alpha)=2$ for any long root $\alpha\in\mathfrak{h}$.

Let $k$ be a positive integer and let $L_{\mathfrak{g}}(k,0)$ be the simple affine VOA associated with $\mathfrak{g}$ at level $k$ (\cite{FZ}).
When the type of $\mathfrak{g}$ is $X_n$, it is also denoted by $L_{X_n}(k,0)$.
Fix a set of simple roots of $\g$.
A dominant integral weight $\Lambda\in\h$ of $\g$ has level $k$ if $(\Lambda|\beta)\le k$ for the highest root $\beta$ of $\g$.
Then $\Irr(L_{\mathfrak{g}}(k,0))=\{L_{\mathfrak{g}}(k,\Lambda)\}$, where $\Lambda$ ranges over dominant integral weights of level $k$ (\cite{FZ}).

Let $\Inn(\g)$ be the inner automorphism group of $\g$ and $\Gamma(\g)$ the Dynkin diagram automorphism group of $\g$.
Then $\Aut(\g)=\Inn(\g):\Gamma(\g)$ (\cite[Section 16.5]{Hu}).
Note that $\Aut(L_\g(k,0))\cong\Aut(\g)$.
The following lemma is immediate:

\begin{lemma}\label{L:autoaff} Let $g\in\Aut(L_\g(k,0))$ and $\Lambda$ a dominant integral weight of level $k$.
\begin{enumerate}[{\rm (1)}]
\item If $g$ is inner, then $L_\g(k,\Lambda)\circ g\cong L_\g(k,\Lambda)$.
\item If $g$ is a Dynkin diagram automorphism, then $L_\g(k,\Lambda)\circ g\cong L_\g(k,g^{-1}(\Lambda))$.
\end{enumerate}
\end{lemma}

\section{VOAs $V_R^+$ associated with indecomposable root lattices $R$}
Let $R$ be an indecomposable root lattice such that $R\not\cong A_1$.
In this section, we prove that $V_R^+$ is isomorphic to the simple affine VOA associated with the semisimple Lie algebra $(V_R^+)_1$ at some positive integral level.
In addition, we study its irreducible modules and automorphisms by using the isomorphism.

\subsection{Isomorphism between $V_R^+$ and the simple affine VOA}
In this subsection, for an indecomposable root lattice $R\not\cong A_1$, we prove that $V_R^+$ is generated by the weight one subspace and we describe the type and level of $V_R^+$ as an affine VOA.

\begin{lemma}\label{L:gen1} Let $R$ be an indecomposable root lattice such that $R\not\cong A_1$.
Then $V_R^+$ is generated by the weight one subspace as a VOA.
\end{lemma}
\begin{proof} Let $\Phi$ be the set of all norm $2$ vectors of $R$.
Then $\Phi$ is a root system; let $\Delta$ be a set of simple roots of $\Phi$.
Let $U$ be the subVOA of $V_R^+$ generated by $(V_R^+)_1$.
Let $\alpha\in \Phi$ and set $x^\alpha=e^{\alpha}+\theta(e^\alpha)\in (V_R^+)_1\subset U$.
Then $x^\alpha_{(-1)}x^\alpha=\lambda\alpha(-1)^2\1$ for some non-zero $\lambda\in \C$, and hence $\alpha(-1)^2\1\in U.$
For $\alpha,\alpha'\in \Phi$ with $(\alpha|\alpha')=-1$, we have $$\alpha(-1)\alpha'(-1)\1=\frac12\left((\alpha+\alpha')(-1)^2-\alpha(-1)^2-\alpha'(-1)^2\right)\1\in U.$$ 
Since $\Phi$ is indecomposable, for any $\beta,\gamma\in\Delta$, we have $\beta(-1)\gamma(-1)\1\in U$.
Hence 
\begin{equation}
\{x^\alpha,\ \beta(-1)\gamma(-1)\1\mid \alpha\in\Phi,\ \beta,\gamma\in\Delta\}\subset U.\label{Eq:U}
\end{equation}
Since the rank of $R$ is at least $2$, we can apply a similar argument as in the proof of \cite[Proposition 12.2.6]{FLM} to our case, and we see that the set \eqref{Eq:U} generates $V_R^+$.
Hence $U=V_R^+$.
\end{proof}

\begin{remark} It follows from \cite[Section 3]{DG98} that $V_{A_1}^+$ is isomorphic to $V_{2A_1}$.
Hence $V_{A_1}^+$ is not generated by $(V_{A_1}^+)_1$ as a VOA.
Indeed, $(V_{2A_1})_1$ generates the proper subVOA $M(1)$.
\end{remark}

\begin{proposition}\label{P:VOAgen1} Let $R$ be an indecomposable root lattice such that $R\not\cong A_1$.
Then $V_R^+$ is isomorphic to the simple affine VOA associated with the semisimple Lie algebra $(V_R^+)_1$ at level $k_R$, where the type of $(V_R^+)_1$ and $k_R$ are given as in Table \ref{T:gQ}; note that we regard $D_2$ and $D_3$ as $A_1^2$ and $A_3$, respectively.
\end{proposition}
\begin{proof}
Recall that the type of the Lie algebra $(V_R)_1$ and that of the root system of $R$ are the same.
By \cite[Section 5, Table II]{He}, the Lie algebra structure of $(V_R^+)_1$ is given as in Table \ref{T:gQ}.
In particular, it is semisimple.
We have mentioned in Section 3.2 that $V_R^+$ is strongly regular.
Hence this proposition follows from Proposition \ref{Prop:posl} and Lemma \ref{L:gen1}.
Note that the level $k_R$ is also given as in Table \ref{T:gQ} (see \cite[Corollary 12.8]{Kac}).
\end{proof}

\begin{table}[bht]
\caption{Lie algebra structure of $(V_R^+)_1$ and level $k_R$} \label{T:gQ}
\begin{tabular}{|c|c|c|c|}
\hline
$R$& $(V_R^+)_1$& $k_R$
\\\hline\hline
$A_2$&$A_1$&$4$
\\\hline
$A_{2n}$ $(n\ge2)$ &$B_n$ &$2$
\\\hline
$A_{2n-1}$ $(n\ge2)$ &$D_n$ &$2$
\\\hline
$D_{2n}$ $(n\ge2)$& $D_{n}^2$& $1$
\\\hline
$D_{2n+1}$ $(n\ge2)$& $B_n^2$&$1$
\\\hline
$E_6$&$C_{4}$ & $1$
\\\hline
$E_7$&$A_7$&$1$
\\\hline
$E_8$&$D_8$ &$1$
\\
 \hline
\end{tabular}
\end{table}

In appendix A, we describe the correspondences between $\Irr(V_R^+)$ and $\Irr(L_\g(k_R,0))$, where $\g=(V_R^+)_1$.

Recall that $\Aut (L_\g(k_R,0))\cong\Aut(\g)$ and that the outer automorphism group $\Out(\g)=\Aut(\g)/\Inn(\g)$ of the semisimple Lie algebra $\g$ is isomorphic to the semidirect product of the diagram automorphism group of $\g$ and the direct product of the symmetric groups on isomorphic simple ideals.

In Table \ref{T:gen}, a set of generators in $\Aut(V_R^+)$ of $\Out(\g)$ is given; we omit the exceptional case $R\cong E_8$ (cf.\ \cite[Corollary 4.4]{Sh06}).
Here, $\lambda_i$ are fundamental weights of $R$, $f_u$ is an element in $\Hom(R,\Z_2)$ (see \eqref{Eq:hom}), $\hat\Gamma$ is a set of lifts in $O(\hat{R})$ of the Dynkin diagram automorphisms of $R$ and $\sigma$ is an extra automorphism of $V_R^+$ (cf.\ Remark \ref{R:DE}).
We adopt the labeling of simple roots as in \cite[Section 11.4]{Hu}.
This table can be verified by direct computation; indeed, the action of $O(\hat{R})$ on $\Irr(V_R^+)$ is described in \cite[Lemma 1.7]{Sh06} (cf.\ \cite[Proposition 2.9]{Sh04}).
On the other hand, the action of $\Aut(L_\g(k_R,0))(\cong\Aut(\g))$ on $\Irr(L_\g(k_R,0))$ is determined by the action of $\Out(\g)$ on the highest weight irreducible $\g$-modules (see Lemma \ref{L:autoaff}).
By using the correspondences between $\Irr(V_R^+)$ and $\Irr(L_\g(k_R,0))$ in Appendix A, one obtain this table.

\begin{table}[bht]
\caption{Generators of $\Out(\g)$ as automorphisms of $V_R^+$} \label{T:gen}
\begin{tabular}{|c|c|c|c|}
\hline
$R$& $\g=(V_R^+)_1$& $\Out(\g)$& generators in $\Aut(V_R^+)$ of $\Out(\g)$
\\\hline\hline
$A_2$, $A_{2n}$ $(n\ge2)$, $E_6$ &$A_1$, $B_n$, $C_4$ &$1$&
\\\hline
$A_3$&$A_1^2$&$\Sym_2$&${f_{\lambda_1}}$\\ \hline
$A_7$&$D_4$&$\Sym_3$&${f_{\lambda_1}}$, $\sigma$\\ \hline
$A_{2n-1}$ ($n=3,\ n\ge5$) &$D_n$ &$\Z_2$& ${f_{\lambda_1}}$
\\\hline
$D_4$&$A_1^4$&$\Sym_4$&${f_{\lambda_1}}$, ${f_{\lambda_3}}$, ${\hat\Gamma}$ \\\hline
$D_{2n}$ $(n=3,\ n\ge5)$& $D_{n}^2$& $\Z_2\wr\Sym_2$& $f_{\lambda_1}$, $f_{\lambda_{2n-1}}$, $\hat\Gamma$
\\\hline
$D_8$&$D_4^2$&$\Sym_3\wr\Sym_2$&${f_{\lambda_1}}$, ${f_{\lambda_7}}$, ${\hat\Gamma}$, $\sigma$ \\\hline
$D_{2n+1}$ $(n\ge2)$& $B_n^2$&$\Sym_2$& $f_{\lambda_1}$
\\\hline
$E_7$&$A_7$&$\Z_2$&$f_{\lambda_1}$
\\\hline
\end{tabular}
\end{table}

\subsection{Automorphism group of $(V_R^+)^{\otimes t}$}

In this subsection, we study the action of $\Aut ((V_R^+)^{\otimes t})$ on $\Irr((V_R^+)^{\otimes t})$.

\begin{lemma}\label{Lem:AutQ0} Let $R$ be an indecomposable root lattice such that $R\not\cong A_1$.
Let $t$ be a positive integer and let $\s$ be a simple ideal of $(V_R^+)_1$.
Then $\Aut((V_R^+)^{\otimes t})$ is the wreath product of $\Aut(\s)$ and the symmetric group on the simple ideals of $(V_R^+)_1$;
$$\Aut ((V_R^+)^{\otimes t})\cong \begin{cases}\Aut (\s)\wr\Sym_t&{\rm if\ }R\cong A_n\ (n\neq3),\ E_6,\ E_7\ {\rm or}\ E_8,\\
\Aut (\s)\wr\Sym_{2t}&{\rm if\ } R\cong A_3\ {\rm or\ }D_{n}\ (n\ge5),\\
\Aut (\s)\wr\Sym_{4t}&{\rm if\ }R\cong D_4.
\end{cases}$$
\end{lemma}
\begin{proof}
By Proposition \ref{P:VOAgen1}, along with Table \ref{T:gQ}, $(V_R^+)^{\otimes t}$ is the tensor product of some copies of $L_\s(k_R,0)$.
By the facts $\Aut (L_\s(k_R,0))\cong\Aut(\s)$ and $\Aut(\s^{\oplus t})\cong\Aut(\s)\wr\Sym_t$, we obtain this lemma.
\end{proof}

Let $\Irr(V_R^+)_0$ denote the subset of $\Irr(V_R^+)$ consisting of isomorphism classes of untwisted type.

\begin{lemma}\label{Lem:AutQQ} Let $R$ be an indecomposable root lattice such that $R\not\cong A_1$.
Then $\Aut ((V_R^+)^{\otimes 2})$ does not preserve $\{M^1\otimes M^2\mid M^i\in \Irr(V_R^+)_0\}\subset\Irr((V_R^+)^{\otimes 2})$ if and only if $R\cong A_3$, $A_7$, $D_n$ $(n\ge4)$ or $E_8$.
\end{lemma}
\begin{proof} 
Assume that $R\cong A_n$ $(n\neq3,7)$, $E_6$ or $E_7$.
By Lemma \ref{L:DE}, $\Aut (V_R^+)$ preserves $\Irr(V_R^+)_0$.
By Lemma \ref{Lem:AutQ0}, $\Aut ((V_R^+)^{\otimes 2})$ also does $\{M^1\otimes M^2\mid M^i\in \Irr(V_R^+)_0\}$.

Assume that $R\cong A_7$, $D_8$ or $E_8$.
By Lemma \ref{L:DE}, $\Aut (V_R^+)$ does not preserve $\Irr(V_R^+)_0$.
Hence $\Aut ((V_R^+)^{\otimes 2})$ also does not preserve $\{M^1\otimes M^2\mid M^i\in \Irr(V_R^+)_0\}$.

Assume that $R\cong A_3$ or $D_{n}$ ($n\ge4$, $n\neq8$).
By Lemma \ref{Lem:AutQ0} and Tables \ref{T:A3}, \ref{T:D2n} and \ref{T:D2n+1} in Appendix A, any element in $\Aut((V_R^+)^{\otimes2})\setminus(\Aut(V_R^+)\wr\Sym_2)$
does not preserve $\{M^1\otimes M^2\mid M^i\in \Irr(V_R^+)_0\}$.
\end{proof}

\begin{lemma}\label{R:t3} Let $Q_i$ $(1\le i\le t)$ be an indecomposable root lattice such that $Q_i\not\cong A_1,E_8$.
Then, there exist $g\in\Aut(\bigotimes_{i=1}^t(V_{Q_i}^+))$ and $M^i\in\Irr(V_{Q_i}^+)_0$ $(1\le i\le t)$ such that $(\bigotimes_{i=1}^t M^i)\circ g\cong\bigotimes_{i=1}^t J^i$ and $J^i\notin \Irr(V_{Q_i}^+)_0$ for all $i$ if and only if $\bigoplus_{i=1}^t Q_i$ is the orthogonal sum of copies of $A_3^2$, $D_n^2$ $(n\ge4,n\neq8)$, $A_7$, $D_8$.
\end{lemma}
\begin{proof}
By Lemmas \ref{L:DE} and \ref{Lem:AutQQ}, the former assertion follows from the latter assertion.

Conversely, we assume the former assertion.
Let $\{1,\dots,t\}=\bigcup_{b\in B}I_b$ be the partition such that $Q_i\cong Q_j$ if and only if $i,j\in I_b$ for some $b\in B$, where $B$ is an index set.
By Table \ref{T:gQ}, simple ideals $\mathfrak{s}_1$ and $\mathfrak{s}_2$ of $(\bigotimes_{i=1}^t(V_{Q_i}^+))_1$ are isomorphic and have the same level if and only if $\mathfrak{s}_1,\mathfrak{s}_2\subset (\bigotimes_{i\in I_b}(V_{Q_i}^+))_1$ for some $b\in B$.
Hence $\Aut(\bigotimes_{i=1}^t(V_{Q_i}^+))\cong\prod_{b\in B}\Aut(\bigotimes_{i\in I_b}(V_{Q_i}^+))$.
Thus it is enough to consider the case where $Q_i\cong Q_j$ for all $i,j$.
By Lemmas \ref{L:DE} and \ref{Lem:AutQ0}, $Q_i$ is neither $A_n$ $(n\neq3,7)$, $E_6$ nor $E_7$.

Assume that $Q_i\cong D_{2n}$ $(n\ge3,\ n\neq4)$ (resp. $A_3$, $D_4$ and $D_{2n+1}$ $(n\ge2)$).
Let $\s$ be a simple ideal of $(V_{Q_i}^+)_1$.
Then the type of $\s$ is $D_n$ (resp. $A_1$, $A_1$ and $B_n$).
It follows from Table \ref{T:D2n} (resp. Tables \ref{T:A3}, \ref{T:D4} and \ref{T:D2n+1}) and the assumptions on $M^i$ and $J^i$ that $M^i$ and $J^i$ have even and odd tensor factors isomorphic to one of $\{[\Lambda_{n-1}],[\Lambda_n]\}$ (resp. $\{[\Lambda_1]\}$, $\{[\Lambda_1]\}$ and $\{[\Lambda_n]\}$) as the irreducible $L_\s(k_{Q_i},0)$-modules, respectively.
Since $\Aut(\s)$ preserves the set $\{[\Lambda_{n-1}],[\Lambda_{n}]\}$ (resp. $\{[\Lambda_1]\}$, $\{[\Lambda_1]\}$ and $\{[\Lambda_n]\}$), $\bigotimes_{i=1}^t J^i$ has even tensor factors isomorphic to one of $\{[\Lambda_{n-1}],[\Lambda_n]\}$ (resp. $\{[\Lambda_1]\}$, $\{[\Lambda_1]\}$ and $\{[\Lambda_n]\}$).
Hence $t$ must be even and we have proved this lemma.
\end{proof}

\begin{remark} By Table \ref{T:gQ}, both $(V_{E_8}^+)_1$ and $(V_{D_{16}}^+)_1$ have a simple ideal of type $D_{8,1}$.
In order to avoid this case, we assume $Q_i\not\cong E_8$ in Lemma \ref{R:t3}, which is enough for our purpose. 
\end{remark}

\section{Automorphism group of the holomorphic VOA $V_L^{{\rm orb}(\theta)}$}
Let $L$ be an even unimodular lattice. 
Let $V_L$ be the lattice VOA associated with $L$ and let $\theta$ be an element of $O(\hat{L})$ such that $\bar{\theta}=-1\in O(L)$.
Let $V_L(\theta)$ be the irreducible $\theta$-twisted $V_L$-module (\cite{FLM}).
Let $V_L^+$ be the fixed-point subspace of $\theta$ and $V_L(\theta)_\Z$ the subspace of $V_L(\theta)$ with integral conformal weights.
Then $V_L^+$ is a full subVOA of $V_L$ and $V_L(\theta)_\Z$ is an irreducible $V_L^+$-module.
Moreover, the $V_L^+$-module $$V=V_L^{\orb(\theta)}:=V_L^+\oplus V_L(\theta)_\Z$$
has a VOA structure as a $\Z_2$-graded simple current extension of $V_L^+$ (\cite{FLM,DGM,EMS}).
Note that $V$ is strongly regular and holomorphic.
Let $z$ be the automorphism of $V$ which acts as $1$ and $-1$ on $V_L^+$ and $V_L(\theta)_\Z$, respectively.
Let $C_{\Aut (V)}(z)$ be the centralizer of $z$ in $\Aut (V)$.
Then $C_{\Aut (V)}(z)=\{g\in\Aut (V)\mid g(V_L^+)=V_L^+\}$.
Hence $C_{\Aut(V)}(z)/\langle z\rangle\subset\Aut(V_L^+)$.

Let $Q$ be the root sublattice of $L$ and let $Q=\bigoplus_{i=1}^t Q_i$ be the orthogonal sum of indecomposable root lattices $Q_i$.
Note that $Q_i$ is $A_\ell$ $(\ell\ge1)$, $D_m$ $(m\ge4)$ or $E_n$ $(n=6,7,8)$.
Now, we consider the following conditions:
\begin{enumerate}[(I)]
\item the rank of $Q$ and $L$ are the same, and it is at least $24$;
\item $Q_i\not\cong A_1$ for all $i$;
\item $Q_i\not\cong E_8$ for all $i$ and $Q$ is not the orthogonal sum of copies of $A_3^2$, $D_{n}^2$ $(n\ge4$, $n\neq8$), $A_7$ and $D_8$.
\end{enumerate}

\begin{lemma}\label{Lem:gstab} 
\begin{enumerate}[{\rm (1)}]
\item If (I) holds, then $V_1\cong\bigoplus_{i=1}^t (V_{Q_i}^+)_1$ and $C_{\Aut(V)}(z)/\langle z\rangle\cong \Aut(V_L^+)$.
\item If (I) and (II) hold, then the subVOA $\langle V_1\rangle$ generated by $V_1$ is isomorphic to $\bigotimes_{i=1}V_{Q_i}^+$.
\item If (I), (II) and (III) hold, then for any  $g\in\Aut(V)$, we have 
$g(V_L^+)=V_L^+$.
\end{enumerate}
\end{lemma}
\begin{proof} Assume (I).
Then the conformal weight of $V_L(\theta)_\Z$ is at least $2$.
Hence (1) follows from $V_1=(V_L^+)_1=(V_Q^+)_1\cong\bigoplus_{i=1}^t (V_{Q_i}^+)_1$.
The assumption (I) also implies that the conformal weights of non-isomorphic irreducible $V_L^+$-modules are different.
Hence any automorphism $g$ of $V_L^+$ satisfies $V_L(\theta)_\Z\circ g\cong V_L(\theta)_\Z$.
Since $V$ is a simple current extension of $V_L^+$, $g$ lifts to an element in $C_{\Aut(V)}(z)$ (\cite[Theorem 3.3]{Sh04}).
Thus $C_{\Aut(V)}(z)/\langle z\rangle\cong \Aut(V_L^+)$.

In addition, we assume (II).
Then (2) follows form (1) and Proposition \ref{P:VOAgen1}.

We further assume (III).
By (2), $\langle V_1\rangle$ is full and strongly regular.
Hence $V$ is the direct sum of finitely many irreducible $\langle V_1\rangle$-submodules.
Let $M$ be an irreducible $\langle V_1\rangle$-submodule of $V_L^+$.
By Lemma \ref{L:decuntw}, $M$ is the tensor product of irreducible $V_{Q_i}^+$-modules of untwisted type.
By Lemma \ref{L:gact} $g(M)$ is isomorphic to $M\circ g_{|\langle V_1\rangle}^{-1}$ as $\langle V_1\rangle$-modules.
By the assumption (III) and Lemma \ref{R:t3}, $g(M)$ contains at least one irreducible $V_{Q_i}^+$-module of untwisted type as a tensor factor.
Clearly, $g(M)$ is an irreducible $\langle V_1\rangle$-submodule of $V_L^+$ or $V_L(\theta)_\Z$.
By Lemma \ref{L:decuntw} again, we obtain $g(M)\subset V_L^+$, and hence $g(V_L^+)=V_L^+$.
\end{proof}

By Lemma \ref{Lem:gstab}, we obtain the following:

\begin{theorem}\label{L:M1} Let $L$ be an even unimodular lattice satisfying (I), (II) and (III).
Let $V=V_L^{\orb(\theta)}$.
Then $\Aut (V)=C_{\Aut (V)}(z)$ and $\Aut (V)/\langle z\rangle\cong \Aut (V_L^+)$.
\end{theorem}

\section{Automorphism groups of $V_N^{\orb(\theta)}$ for Niemeier lattices}

Let $N$ be a Niemeier lattice and let $Q$ be its root sublattice.
Let $V=V_N^{\orb(\theta)}$ be the holomorphic VOA given in the previous section.
If $Q=0$, then $N$ is isometric to the Leech lattice and $V$ is isomorphic to the moonshine VOA whose automorphism group is the Monster (\cite{FLM}); hence we assume $Q\neq0$.
We also assume that the $-1$-isometry of $N$ does not belong to the Weyl group of $Q$; otherwise, $\theta$ is inner, and $V$ is isomorphic to a Niemeier lattice VOA (cf.\ \cite{DGM}).
Then $Q$ is one of $14$ root lattices in Lemma \ref{L:-1G0} and (I) and (II) in the previous section hold.
In this section, we determine the groups $K(V)$, $\Out_1(V)$ and $\Out(V_2)$, which proves Theorem \ref{T:main} in Introduction.

\begin{remark}\label{R:ext} 
If $Q$ is neither $A_3^8$ nor $D_5^2A_7^2$, then $Q$ also satisfies the assumption (III) in Section 5, and by Theorem \ref{L:M1}, $\Aut (V)=C_{\Aut (V)}(z)$.
If $Q$ is $A_3^8$ or $D_5^2A_7^2$, then $N$ is constructed from the doubly even self-dual binary code $d_6^4$ or $d_{10}e_7^2$ of length $24$ by the same manner as the Leech lattice, respectively (see \cite[Fig.\ 3]{DGM}); $V$ has an extra automorphism not in $C_{\Aut (V)}(z)$ (cf.\ \cite[Chapter 11]{FLM}), and $\Aut (V)$ is greater than $C_{\Aut (V)}(z)$.
\end{remark}

\begin{remark} For a Niemeier lattice $N$ with the root sublattice $Q\neq0$, the groups $K(V_N)$ and $\Out(V_N)$ for the lattice VOA $V_N$ are determined by the following way. 
Since $\langle (V_N)_1\rangle= V_Q$ and $V_N$ is a simple current extension of $V_Q$ graded by $N/Q$, we see that $K(V_N)\cong (N/Q)^*\cong N/Q$ and $K(V_N)\subset\Inn(V_N)$.
By \cite{DN}, $\Out(V_N)=\Aut(V_N)/\Inn(V_N)\cong O(N)/W(Q)$, which is the automorphism group of the glue code $N/Q$ (\cite[Chapter 16]{CS}).
\end{remark}

\subsection{The subgroup $K(V)$}
In this subsection, we determine $K(V)$, the subgroup of $\Aut(V)$ which acts trivially on $V_1$.
It follows from $V_1\subset V_N^+$ that $K(V)$ contains the automorphism $z$ of $V$ defined in Section 5.
Let $Q=\bigoplus_{i=1}^tQ_i$ be the orthogonal sum of indecomposable root lattices.

\begin{lemma}\label{L:KV}
\begin{enumerate}[{\rm (1)}]
\item $K(V)/\langle z\rangle\subset O(\hat{N})/\langle\theta\rangle$;
\item $(K(V)/\langle z\rangle)\cap\Hom(N,\Z_2)=\{f_u\mid u\in (N\cap 2Q^*)/2N\}$;
\item $\overline{K(V)/\langle z\rangle}=\{f\in O(N)\mid f=\pm1\ {\rm on}\ Q_i,\  1\le \forall i\le t\}/\langle-1\rangle$, where $\bar{\ }$ is the map $O(\hat{N})/\langle\theta\rangle\to O(N)/\langle-1\rangle$ in the exact sequence \eqref{Exact2}.
\end{enumerate}
\end{lemma}
\begin{proof}
Let $g\in K(V)$.
Since $g$ acts trivially on $V_1$, for any irreducible $\langle V_1\rangle$-submodule $M$ of $V$, we have $M\circ g\cong M$.
It follows from Lemma \ref{Lem:gstab} (2) that $\langle V_1\rangle\cong \bigotimes_{i=1}^t V_{Q_i}^+$.
By Lemma \ref{L:decuntw}, we have $g(V_N^+)=V_N^+$ and $g(V_N(\theta)_\Z)=V_N(\theta)_\Z$.
Hence $g\in C_{\Aut(V)}(z)$.
Let $g_0$ denote the restriction of $g$ to $V_N^+$. 
Since the four (non-isomorphic) irreducible $V_N^+$-modules have different conformal weights, we have $V_N^-\circ {g_0}\cong V_N^-$.
Note that $V_N=V_N^+\oplus V_N^-$ is a $\Z_2$-graded simple current extension of $V_N^+$.
Hence, there exists a lift  $\hat{g}\in C_{\Aut(V_N)}(\theta)$ of $g_0$ by \cite[Theorem 3.3]{Sh04}.
Since any irreducible $\langle V_1\rangle$-submodule of $V_Q$ is a simple current module, its multiplicity is one, and $\hat{g}$ acts by a scalar on it.
In particular $\hat{g}$ preserves the Cartan subalgebra $\h$ of $V_Q$.
Recall from \cite[Section 2.4]{DN} that the stabilizer of $\h$ in $C_{\Aut (V_L)}(\theta)$ is $O(\hat{L})$.
Hence, $\hat{g}\in O(\hat{N})$, and $g_{0}\in O(\hat{N})/\langle\theta\rangle$.
Since $\{h\in \Aut(V)\mid h=id\ {\rm on}\ V_N^+\}=\langle z\rangle$, we have $K(V) /\langle z\rangle\subset O(\hat{N})/\langle\theta\rangle$, which proves (1).

The assertion (1) shows that $K(V)/\langle z\rangle=\{h\in O(\hat{N})/\langle\theta\rangle\mid h=id\ {\rm on}\ (V_N^+)_1\}$.
The assertion (2) follows from $V_1\cong\bigoplus_{i=1}^t(V_{Q_i}^+)_1$ and the explicit action of $\Hom(N,\Z_2)$ in \eqref{Eq:hom}.

Set $F=\{f\in O(N)\mid f=\pm1\ {\rm on}\ Q_i,\  1\le \forall i\le t\}/\langle-1\rangle.$
Clearly $\overline{K(V)/\langle z\rangle}\subset F$.
By Remark \ref{L:-1}, if $Q$ is $A_5^4D_4$, $A_9^2D_6$ or $A_{17}E_7$, then $F\cong\Z_2$ and $K(V)/\langle z\rangle$ contains a lift of the $-1$-isometry of the root lattice $D_4$, $D_6$ or $E_7$, respectively, and $\overline{K(V)/\langle z\rangle}=F$; otherwise $F=1$. Thus (3) holds.
\end{proof}

\begin{proposition}\label{P:K} Let $P^\vee$ be the coweight lattice of the Lie algebra $\g=V_1$.
Then $K(V)=\{\sigma_x\mid x\in P^\vee\}$.
Moreover, the group structure of $K(V)$ is given as in Table \ref{T:Q}.
\end{proposition}
\begin{proof} It is clear that $\sigma_x\in K(V)$ for all $x\in P^\vee$.
By Lemma \ref{L:KV} and the properties of $N$ and $O(N)$ in \cite[Chapters 16 and 18]{CS}, the order of $K(V)$ is determined as in Table \ref{T:K(V)}.
By using Tables \ref{T:A2} to \ref{T:E7} in Appendix A, along with the generators of the glue code $N/Q$ in \cite[Section 18.4]{CS}, we can describe the module structure of $V_N^+$ for simple affine VOAs.
Then for $x\in P^\vee$, we know the action of $\sigma_x$ on the irreducible $\langle V_1\rangle$-submodule of $V_N^+$ as an element of $O(\hat{N})$; we see that $\{\sigma_x\mid x\in P^\vee\}$ contains generators of $K(V)$ (see Lemma \ref{L:KV} and Table \ref{T:K(V)}).
Thus the group structure of $K(V)$ is determined as in Table \ref{T:Q}.
Note that the vectors for the case $A_3^8$ in Table \ref{T:K(V)} are described with respect to some specified coordinate (see Section \ref{S:A38} below).
\end{proof}
\begin{corollary}\label{C:KV} The group $K(V)$ is contained in $\Inn(V)$.
\end{corollary}

\begin{Small}
\begin{table}[bht]
\caption{Vectors in $P^\vee$ s.t. the associated inner automorphisms generate $K(V)$} \label{T:K(V)}
\begin{tabular}{|c|c|c|c|c|c|}
\hline
$Q$&$V_1$&$|K(V)|$&$z$&$(K(V)/\langle z\rangle)\cap{\Hom} (N,\Z_2)$&${\overline{K(V)/\langle z\rangle}}$\\ \hline
$A_2^{12}$&$A_{1,4}^{12}$&$2$&$(\Lambda_1,0^{11})$&$1$&$1$\\ \hline
$A_3^8$&$A_{1,2}^{16}$&$2^5$&$(\Lambda_1,\Lambda_1,0^{14})$& $(\Lambda_1,0^{15})$, $(0^2,\Lambda_1,0^{13})$,&$1$\\
&&&&$(0^4,\Lambda_1,0^{11})$, $(0^8,\Lambda_1,0^7)$
& \\ \hline
$A_4^6$&$B_{2,2}^6$&$2$&$(\Lambda_1,0^5)$&$1$&$1$\\ \hline
$A_5^4D_4$&$A_{3,2}^4A_{1,1}^4$&$2^5$&$(2\Lambda_1,0^7)$&$(\Lambda_1,0^7)$, $(0,\Lambda_1,0^6)$, $(0^2,\Lambda_1,0^5)$&$(0^4,\Lambda_1^4)$\\ \hline
$A_6^4$&$B_{3,2}^{4}$&$2$&$(\Lambda_1,0^3)$&$1$&$1$\\ \hline
$A_7^2D_5^2$&$D_{4,2}^2B_{2,1}^{4}$&$2^3$&$(\Lambda_1,0^5)$&$(\Lambda_4,0^5),(0^4,\Lambda_1,0)$&$1$\\ \hline
$A_8^3$&$B_{4,2}^{3}$&$2$&$(\Lambda_1,0,0)$&$1$&$1$\\ \hline
$A_9^2D_6$&$D_{5,2}^{2}A_{3,1}^{2}$&$2^4$&$(0^2,2\Lambda_3,2\Lambda_3)$&$(0^2,\Lambda_3,\Lambda_3),(0^2,2\Lambda_3,0)$&$(0^2,\Lambda_3,0)$\\ \hline
$E_6^4$&$C_{4,1}^{4}$&$2$&$(\Lambda_4,0^3)$&$1$&$1$\\ \hline
$A_{11}D_7E_6$&${D_{6,2}}B_{3,1}^{2}C_{4,1}$&$2^2$&$(\Lambda_1,0,0)$&$(\Lambda_6,0,0)$&$1$\\ \hline
$A_{12}^2$&$B_{6,2}^{2}$&$2$&$(\Lambda_1,0)$&$1$&$1$\\ \hline
$A_{15}D_9$&${D_{8,2}}B_{4,1}^{2}$&$2^2$&$(\Lambda_1,0,0)$&$(\Lambda_8,0,0)$&$1$\\ \hline
$A_{17}E_7$&${D_{9,2}}{A_{7,1}}$&$2^3$&$(0,\Lambda_4)$&$(0,\Lambda_2)$&$(0,\Lambda_1)$\\ \hline
$A_{24}$&${B_{12,2}}$&$2$&$(\Lambda_1)$&$1$&$1$\\ \hline
\end{tabular}
\end{table}
\end{Small}

\subsection{The groups $\Out_1(V)$ and $\Out_2(V)$}
In this subsection, we determine the groups $\Out_1(V)$ and $\Out_2(V)$.
Let $V_1=\bigoplus_{i=1}^s\g_i$ be the direct sum of simple ideals.

\begin{lemma}\label{L:O1}
If $Q_i\not\in\{A_{2n-1}, D_{2n}, E_7, E_8\mid n\ge3\}$ for all $1\le i\le t$, then $O_1(V)=1$.
\end{lemma}
\begin{proof} By Table \ref{T:gQ}, for any simple ideal of $V_1$, its type is neither $A_n$ $(n\ge2)$, $D_n$ ($n\ge4$) nor $E_6$, which shows that $\Out_1(V)=1$.
\end{proof}

\begin{lemma}\label{L:OVL+2}
The group $\Aut(V_N^+)$ preserves the set $\{(V_{Q_i}^+)_1\mid 1\le i\le t\}$ of semisimple ideals of $V_1$.
Moreover, its action is the same as that of $G_2(N)$ on $\{Q_i\mid 1\le i\le t\}$.
\end{lemma}
\begin{proof} By \eqref{Exact1} (see also Section 2.1), $O(\hat{N})$ acts on $\{(V_{Q_i})_1\mid 1\le i\le t\}$ as the permutation group $G_2(N)$ on $\{Q_i\mid 1\le i\le t\}$.
Hence $O(\hat{N})/\langle\theta\rangle$ also acts on $\{(V_{Q_i}^+)_1\mid 1\le i\le t\}$ by the same manner.
On the other hand, $\Inn(V_N)$ preserves $V_{Q_i}$ for all $i$.
Thus we obtain this lemma by Lemma \ref{Lem:AutVL+}.
\end{proof}

\begin{lemma}\label{L:O2}
\begin{enumerate}[{\rm (1)}]
\item $\Out(V_N^+)\subset \Out(V)$.
\item If $Q$ is neither $A_3^8$ nor $A_7^2D_5^2$, then $\Out(V)=\Out(V_N^+)$.
\item If $Q_i\notin\{A_3,D_{n}\mid n\ge4\}$ for all $1\le i\le t$, then $\Out_2(V)\cong G_2(N)$.
\end{enumerate}
\end{lemma}
\begin{proof} The assertion (1) follows from Lemma \ref{Lem:gstab} (1), $V_1=(V_N^+)_1$ and $z\in K(V)$.

If $Q$ is neither $A_3^8$ nor $A_7^2D_5^2$, then by Theorem \ref{L:M1} (cf.\ Remark \ref{R:ext}), $\Aut(V)/\langle z\rangle\cong\Aut(V_N^+)$.
Hence (2) holds.

Assume that $Q_i\notin\{A_3,D_{n}\mid n\ge4\}$ for all $1\le i\le t$.
By Table \ref{T:gQ}, every $(V_{Q_i}^+)_1$ is a simple ideal of $V_1$.
Hence $\{(V_{Q_i}^+)_1\mid 1\le i\le t\}=\{\g_i\mid 1\le i\le s\}$ and $s=t$.
By (2) and Lemma \ref{L:OVL+2}, we have $\Out_2(V)\cong \Out_2(V_N^+)\cong G_2(N)$.
\end{proof}

Combining Lemmas \ref{L:O1} and \ref{L:O2} (3), we obtain the following:

\begin{proposition}\label{P:OV} If $Q\cong A_2^{12}$, $A_4^6$, $A_6^4$, $A_8^3$, $E_6^4$, $A_{12}^2$ or $A_{24}$, then $O_1(V)=1$ and $O_2(V)\cong G_2(N)$.
In particular, the group structures of $O_1(V)$ and $O_2(V)$ are given as in Table \ref{T:Q}.
\end{proposition}

Let us consider the remaining cases.
Set $N_0=N\cap(Q/2)$.

\begin{lemma}\label{L:N0}
\begin{enumerate}[{\rm (1)}]
\item $V_{N_0}^+$ is the direct sum of all simple current $\langle V_1\rangle$-submodules of $V_N^+$.
\item $\Aut(V)$ preserves $V_{N_0}^+$.
\end{enumerate}
\end{lemma}
\begin{proof}
(1) follows from Lemma \ref{L:decuntw} and Proposition \ref{P:ADL}.

If $Q$ is neither $A_3^8$ nor $A_7^2D_5^2$, then by Theorem \ref{L:M1} (cf. Remark \ref{R:ext}), $\Aut(V)$ preserves $V_N^+$.
By (1), Lemmas \ref{L:gcon} (3) and \ref{L:gact}, we obtain (2).

If $Q$ is $A_3^8$ (resp. $A_7^2D_5^2$), then by Table \ref{T:A3} (resp. Tables \ref{T:A2n-1} and \ref{T:D2n+1}), $V_{N_0}^+$ is the sum of all irreducible $\langle V_1\rangle$-submodules of $V$ isomorphic to the tensor products of $L_{A_1}(2,0)$ and $L_{A_1}(2,2\Lambda_1)$. (resp.  $L_{D_4}(2,0)$, $L_{D_4}(2,2\Lambda_i)$ $(i=1,3,4)$, $L_{B_2}(1,0)$ and $L_{B_2}(1,\Lambda_1)$).
By Lemmas \ref{L:gact} and \ref{L:autoaff}, $\Aut(V)$ preserves $V_{N_0}^+$.
\end{proof}

Let $M$ be an irreducible $\langle V_1\rangle$-submodule of $V_{N_0}^+$.
Since $M$ is a simple current module (see Lemma \ref{L:N0} (1)), the multiplicity of $M$ in $V_{N_0}^+$ is one.
Moreover, for irreducible $\langle V_1\rangle$-submodules $M^1$ and $M^2$ of $V_{N_0}^+$, the fusion product $M^1\boxtimes M^2$ is an irreducible $\langle V_1\rangle$-module and it appears as a $\langle V_1\rangle$-submodule of $V_{N_0}^+$.

Let $S_i$ be the set of (the isomorphism classes of) simple current $L_{\g_i}(k_i,0)$-modules, where $k_i$ is the level of $\g_i$ in $V$.
Then $S_i$ has an abelian group structure under the fusion product.
By  \cite[Theorem 2.26]{Li01}, $S_i\cong \Z_{n+1},\Z_2,\Z_2,\Z_2^2,\Z_4,\Z_3$ or $\Z_2$ if the type of $\g_i$ is $A_n$ $(n\ge1)$, $B_n$ $(n\ge2)$, $C_n$ $(n\ge2)$, $D_{2n}$ ($n\ge2$), $D_{2n+1}$ ($n\ge2$), $E_6$ or $E_7$, respectively.
Let $S_N=\prod_{i=1}^s S_i$ be the direct product of the groups $S_i$.
We often view a simple current $\langle V_1\rangle$-module as an element of $S_N$ via the map $\bigotimes_{i=1}^sM^i\mapsto (M^1,\dots,M^s)$.

Let $\{1,2,\dots,s\}=\bigcup_{b\in B} I_b$ be the partition such that $\g_i\cong\g_j$ if and only if $i,j\in I_b$ for some $b\in B$, where $B$ is an index set.
By Lemma \ref{L:autoaff} (2) and the explicit description of $S_i$, 
the Dynkin diagram automorphism group  $\Gamma(\g_i)\subset\Aut(L_{\g_i}(k_i,0))$ acts faithfully on $S_i$.
The automorphism group $\Aut(S_N)$ of $S_N$ is defined to be $(\prod_{i=1}^s\Gamma(\g_i)):(\prod_{b\in B} \Sym_{|I_b|})$, where the symmetric group $\Sym_{|I_b|}$ acts naturally on $\prod_{i\in I_b}S_i$.

Let $C_N$ be the subgroup of $S_N$ consisting of all the irreducible $\langle V_1\rangle$-submodule of $V_{N_0}^+$.
The automorphism group $\Aut(C_N)$ of $C_N$ is defined to be the subgroup of $\Aut(S_N)$ stabilizing $C_N$.
Let $\Aut_1(C_N)=(\prod_{i=1}^s\Gamma(\g_i))\cap \Aut(C_N)$ and $\Aut_2(C_N)=\Aut(C_N)/\Aut_1(C_N)$.
Then $\Aut_1(C_N)$ is the subgroup of $\Aut(C_N)$ stabilizing every $S_i$, and $\Aut_2(C_N)$ acts faithfully on $\{S_i\mid 1\le i\le s\}$ as a permutation group.

Since $\Out(\g_i)$ acts faithfully on $S_i$, so does $\Out(V)$ on $S_N$.
Hence $\Out(V)\subset\Aut(S_N)$.
In addition, by Lemma \ref{L:N0} (2), we obtain the following:
\begin{lemma}\label{L:AutCQ}
$\Out(V)$ is a subgroup of $\Aut(C_N)$.
\end{lemma}
\begin{remark} In general, $\Out(V)$ is not equal to $\Aut(C_N)$.
For example, if $Q= A_2^{12}$, then $\Out(V)\subsetneq\Aut(C_N)$; indeed, $N_0=N\cap(Q/2)=Q$ and $C_N$ is the set of all even weight codewords of $\prod_{i=1}^{12}S_i\cong \Z_2^{12}$.
Hence $\Aut(C_N)\cong\Sym_{12}$.
On the other hand, $\Out(V)\cong M_{12}$ by Proposition \ref{P:OV}.
\end{remark}

\begin{proposition}\label{P:O12} Let $N$ be a Niemeier lattice whose root sublattice is $A_3^8$, $A_5^4D_4$, $A_7^2D_5^2$, $A_9^2D_6$, $A_{11}D_7E_6$, $A_{15}D_9$ or $A_{17}E_7$ and let $V=V_N^{\orb(\theta)}$.
Then $\Out_i(V)=\Aut_i(C_N)$ for $i=1,2$ and $\Out(V)=\Aut(C_N)$.
In particular, the group structures of $O_1(V)$ and $O_2(V)$ are given as in Table \ref{T:Q}.
\end{proposition}
By Propositions \ref{P:K}, \ref{P:OV} and \ref{P:O12}, we obtain Theorem \ref{T:main}, the main theorem of this article.

In the following subsections, we will prove Proposition \ref{P:O12} by a case-by-case analysis.
By Lemma \ref{L:AutCQ}, it suffices to prove $\Out_i(V)\supset\Aut_i(C_N)$ for $i=1,2$.
In order to describe $C_N$ explicitly, we denote elements of $\Z_n$ and $\Z_2\times\Z_2$ by their representatives in $\{0,1,\dots,n-1\}$ and $\{0,1,w,1+w\}$, respectively.
For the action of $\Hom(N,\Z_2)$ on $S_N$, see Lemma \ref{L:Hommod} and tables in Appendix A (cf.\ Table \ref{T:gen}).
Note that $\lambda_i$ denote the fundamental weights for indecomposable root lattices.

\subsubsection{Case $Q\cong A_3^8$}\label{S:A38}
In this case, the type of $V_1$ is $A_{1,2}^{16}$.
Then $S_N\cong\Z_2^{16}$ and $\Out_1(V)=\Aut_1(C_N)=1$.

By the description of the glue code $N/Q$ in \cite[Section 18.4, III]{CS} and Table \ref{T:A3}, $C_N$ is equivalent to the second order Reed-Muller code of length $16$ (see \cite[Chapter 13]{MS} for its definition).
It is well-known (\cite[Section 13.9]{MS}) that $\Aut_2(C_N)\cong \Z_2^4:L_4(2)$.

Let us prove $\Out(V)\supset\Aut(C_N)$.
The subgroup $\Hom(N,\Z_2)=\{f_u\mid u\in N/2N\}$ of $\Aut(V_N^+)$ acts on $\{S_i\mid 1\le i\le 16\}$ as an elementary abelian $2$-subgroup $X\subset\Sym_{16}$ of order $2^4$; indeed, $\{f_u\mid u\in (N\cap 2Q^*)/2N\}$ acts trivially on $C_N$, and $X\cong N/(N\cap 2Q^*)\cong \Z_2^4$.
Note that $X$ preserves every $V_{Q_i}^+$.
In addition, by Lemma \ref{L:OVL+2}, $\Aut(V_N^+)$ acts on $\{(V_{Q_i}^+)_1\mid 1\le i\le 8\}$ as $G_2(N)\cong \Z_2^3:L_3(2)$.
Combining these two actions, we obtain a subgroup of $\Out(V)$ of shape $\Z_2^4:(\Z_2^3:L_3(2))$, which is a maximal subgroup of $\Z_2^4:L_4(2)$.
In addition, $V$ has an extra automorphism (see Remark \ref{R:ext}) not in the maximal subgroup above.
Hence $\Out_2(V)=\Aut_2(C_N)\cong \Z_2^4:L_4(2)$.

\subsubsection{Case $Q\cong A_5^4D_4$}\label{S:A54D4}
In this case, the type of $V_1$ is $A_{3,2}^4A_{1,1}^4$.
Then $S_N\cong \Z_4^4\times\Z_2^4$.
For the explicit description of the glue code $N/Q$, see \cite[Section 18.4 XVI]{CS}.

By the generator of $N/Q$ and Tables \ref{T:A2n-1} and \ref{T:D4}, $C_N$ is generated by
\begin{align*}
(2,0,0,0,\ 1,1,1,1),\ (1,1,0,0,\ 1,1,0,0),\ (0,1,0,1,\ 0,1,0,1),\ (1,1,1,1,\ 0,0,0,0).
\end{align*}
Here the type of $\g_i$ is $A_{3,2}$ (resp. $A_{1,1}$) if $1\le i\le 4$ (resp. $5\le i\le 8$).
It is easy to see that $\Aut_1(C_N)(\cong\Z_2)$ is generated by $-1$ on $\prod_{i=1}^4S_i$ and that $\Aut_2(C_N)$ has the shape $\Z_2^4:\Sym_3$, where the normal subgroup $\Z_2^4$ is the direct product of the Klein four-subgroups of $\Sym_4$ on $\{S_i\mid 1\le i\le 4\}$ and $\{S_i\mid 5\le i\le 8\}$, and $\Sym_3$ acts diagonally on $\{S_i\mid 1\le i\le 4\}$ and $\{S_i\mid 5\le i\le 8\}$ as a subgroup of $\Sym_4$.

Let us prove that $\Out_i(V)\supset\Aut_i(C_N)$ for $i=1,2$.
For $x\in N$, we use the coordinate $x=(x_1,\dots,x_5)\in(A_5^*)^4D_4^*\cap N$.
The automorphism $f_{(\lambda_1,\lambda_1,\lambda_1,\lambda_1,0)}$ generates a subgroup of $\Out_1(V)$ of order $2$, and hence $\Out_1(V)=\Aut_1(C_N)\cong\Z_2$.
The automorphisms $f_{(0,\lambda_1,2\lambda_1,-\lambda_1,\lambda_3)}$ and $f_{(3\lambda_1,3\lambda_1,0,0,\lambda_1)}$ generates the Klein four-group of $\Sym_4$ on $\{S_i\mid 5\le i\le 8\}$.
By Lemma \ref{L:OVL+2},  $\Aut(V_N^+)$ acts on $C_N$ as a permutation group $G_2(N)\cong\Sym_4$; it acts on $\{S_i\mid 1\le i\le 4\}$ as $\Sym_4$ but it does on $\{S_i\mid 5\le i\le 8\}$ as the quotient group $\Sym_3$ of $\Sym_4$.
Thus $\Out_2(V)$ contains a subgroup of shape $\Z_2^4:\Sym_3$, and $\Out_2(V)=\Aut_2(C_N)$.

\subsubsection{Case $Q\cong A_7^2D_5^2$}\label{S:A72D52}
In this case, the type of $V_1$ is $D_{4,2}^2B_{2,1}^4$.
Then $S_N\cong (\Z_2^2)^2\times\Z_2^4$.
For the explicit description of the glue code $N/Q$, see \cite[Section 18.4 XVII]{CS}.

By the generator of $N/Q$ and Tables \ref{T:A2n-1} and \ref{T:D2n+1}, $C_N$ is generated by
\begin{align*}
(1,1,\ 0,0,0,0),\ (w,w,\ 0,0,0,0),\ (0,0,\ 1,1,1,1),\ (1,0,\ 1,1,0,0),\ (w,0,\ 1,0,1,0).\label{C:A72}
\end{align*}
Here, the type of $\g_i$ is $D_{4,2}$ (resp. $B_{2,1}$) if $1\le i\le2$ (resp. $3\le i\le 6$).
It is easy to see that $\Aut_1(C_N)=1$ and that $\Aut_2(C_N)$ is $\Sym_2\times \Sym_4$, where $\Sym_2$ and $\Sym_4$ act on $\{S_1,S_2\}$ and $\{S_i\mid 3\le i\le 6\}$ as the symmetric group, respectively.

Clearly, $\Out_1(V)=1$.
Let us prove that $\Out_2(V)\supset\Aut_2(C_N)$.
For $x\in N$, we use the coordinate $x=(x_1,\dots,x_4)\in (A_7^*)^2(D_5^*)^2\cap N$.
The automorphism $f_{(3\lambda_1,\lambda_1,\lambda_1,0)}$ (resp. $f_{(5\lambda_1,\lambda_1,0,\lambda_1)}$) act as the order $2$ permutation on $\{S_3,S_4\}$ (resp. $\{S_5,S_6\}$).
By Lemma \ref{L:OVL+2}, $\Aut(V_N^+)$ acts on $\{S_1,S_2\}$ and $\{\{S_3,S_4\},\{S_5,S_6\}\}$ as the symmetric group $G_2(N)\cong\Sym_2\times\Sym_2$.
Combining these actions, we obtain a subgroup of $\Out_2(V)$ of order $2^4$, which is a maximal subgroup of $\Aut_2(C_N)\cong \Sym_2\times \Sym_4$.
In addition, $V$ has an extra automorphism (see Remark \ref{R:ext}) not in the maximal subgroup above.
Thus we have $\Out_2(V)=\Aut_2(C_N)\cong \Sym_2\times \Sym_4$.   

\subsubsection{Case $Q\cong A_9^2D_6$}\label{S:A92D6}
In this case, the type of $V_1$ is $D_{5,2}^2A_{3,1}^2$.
Then $S_N\cong \Z_4^2\times\Z_4^2$.
For the explicit description of the glue code $N/Q$, see \cite[Section 18.4, XVIII]{CS}.

By the generator of $N/Q$ in \cite{CS} and Tables \ref{T:A2n-1} and \ref{T:D2n}, $C_N$ is generated by
\begin{align*}
(1,1,\ 2,0),\ (1,0,\ 1,1).
\end{align*}
Here, the type of $\g_i$ is $D_{5,2}$ (resp. $A_{3,1}$) if $1\le i\le2$ (resp. $3\le i\le 4$).
It is easy to see that $\Aut_1(C_N)\cong\Z_2$ is generated by $-1$ on $C_N$ and that $\Aut_2(C_N)\cong\Sym_2\times \Sym_2$.

Let us prove that $\Out_i(V)\supset\Aut_i(C_N)$ for $i=1,2$.
For $x\in N$, we use the coordinate $x=(x_1,x_2,x_3)\in (A_9^*)^2D_6^*$.
The automorphism $f_{(\lambda_5,\lambda_5,\lambda_1)}$ acts on $C_N$ by $-1$, and hence $\Out_1(V)= \Aut_1(C_N)\cong\Z_2$.
The automorphism $f_{(\lambda_1,\lambda_2,\lambda_5)}$ permutes $S_3$ and $S_4$.
In addition, $\Aut(V_N^+)$ acts on the permutation group $G_2(N)\cong \Sym_2$ on $\{S_1,S_2\}$.
Hence we have $\Out_2(V)=\Aut_2(C_N)\cong\Sym_2\times\Sym_2$.

\subsubsection{Case $Q\cong A_{11}D_7E_6$}\label{S:A11D7E6}
In this case, the type of $V_1$ is $D_{6,2}B_{3,1}^2C_{4,1}$.
Then $S_N\cong (\Z_2^2)\times(\Z_2)^2\times \Z_2$.
For the explicit description of the glue code $N/Q$, see \cite[Section 18.4, XXIII]{CS}.

By the generator of $N/Q$ and Tables \ref{T:A2n-1}, \ref{T:D2n+1} and \ref{T:E6}, $C_N$ is generated by
\begin{align*}
(1,\ 1,1,\ 0),\ (1,\ 0,0,\ 1),\ (w,\ 1,0,\ 0).
\end{align*} 
Here, the types of $\g_1$, $\g_2$, $\g_3$, $\g_4$ are $D_{6,2}$, $B_{3,1}$, $B_{3,1}$ and $C_{4,1}$, respectively.
It is easy to see that $\Aut_1(C_N)=1$ and that $\Aut_2(C_N)\cong\Sym_2$.
Hence $\Out_1(V)=1$.
The automorphism $f_{(\lambda_1,\lambda_1,\lambda_1)}$ permutes $S_2$ and $S_3$, where $(\lambda_1,\lambda_1,\lambda_1)\in A_{11}^*D_7^*E_6^*\cap N$.
Thus we have $\Out_2(V)=\Aut_2(C_N)\cong\Sym_2$.

\subsubsection{Case $Q\cong A_{15}D_9$}\label{S:A15D9}
In this case, the type of $V_1$ is $D_{8,2}B_{4,1}^2$.
Then $S_N\cong (\Z_2^2)\times(\Z_2)^2$.
For the explicit description of the glue code $N/Q$, see \cite[Section 18.4, XIX]{CS}.

By the generator of $N/Q$ and Tables \ref{T:A2n-1} and \ref{T:D2n+1}, $C_N$ is generated by
\begin{align*}
(1,\ 1,1),\ (w,\ 0,0).
\end{align*} 
Here, the types of $\g_1,\g_2,\g_3$ are $D_{8,2}$, $B_{4,1}$, $B_{4,1}$, respectively.
It is easy to see that $\Aut_1(C_N)=1$ and that $\Aut_2(C_N)\cong\Sym_2$.
Hence $\Out_1(V)=1$.
The automorphism $f_{(\lambda_2,\lambda_1)}$ permutes $S_2$ and $S_3$, where $(\lambda_2,\lambda_1)\in A_{15}^*D_9^*\cap N$.
Thus we have $\Out_2(V)=\Aut_2(C_N)\cong\Sym_2$.

\subsubsection{Case $Q\cong A_{17}E_7$}\label{S:A17E7}
In this case, the type of $V_1$ is $D_{9,2}A_{7,1}$.
Then $S_N\cong\Z_4\times\Z_8$ and $\Out_2(V)=\Aut_2(C_N)=1$.
For the explicit description of the glue code $N/Q$, see \cite[Section 18.4, XXII]{CS}.

By the generator of $N/Q$ and Tables \ref{T:A2n-1} and \ref{T:E7}, $C_N$ is generated by
\begin{align*}
(1,\ 2).
\end{align*} 
Here, the types of $\g_1$ and $\g_2$ are $D_{9,2}$ and $A_{7,1}$, respectively.
It is easy to see that $\Aut(C_N)=\Aut_1(C_N)\cong\Z_2$ and it is generated by $-1$ on $C_N$.
The automorphism $f_{(\lambda_3,\lambda_1)}$ generates $\Aut_1(C_N)$, where $(\lambda_3,\lambda_1)\in A_{17}^*E_7^*\cap N$.
Hence $\Out_1(V)=\Aut_2(C_N)\cong\Z_2$.

\begin{remark} The subgroup $C_N$ of $S_N$ is called ``Glue" in \cite[Table 1]{Sc93}.
\end{remark}

Combining Theorem \ref{L:M1}, Remark \ref{R:ext}, Corollary \ref{C:KV} and the arguments in Sections \ref{S:A38} and \ref{S:A72D52}, we obtain the following corollary:

\begin{corollary}\label{C:AutV} Let $N$ be a Niemeier lattice with root sublattice $Q$ and let $V=V_N^{\orb(\theta)}$.
\begin{enumerate}[{\rm (1)}]
\item If $Q\cong A_2^{12}$, $A_4^6$, $A_5^4D_4$, $A_6^4$, $A_8^3$, $A_9^2D_6$, $E_6^4$, $A_{11}D_7E_6$, $A_{12}^2$, $A_{15}D_9$, $A_{17}E_7$ or $A_{24}$, then $z$ is a central element of $\Aut(V)$ and $\Aut(V)/\langle z\rangle\cong\Aut(V_N^+)$.
\item If $Q\cong A_3^8$ or $A_7^2D_5^2$, then $\Aut(V)$ is generated by $C_{\Aut(V)}(z)$ and an extra automorphism in \cite{FLM}.
\end{enumerate}
\end{corollary}

\appendix\section{Correspondence between $\Irr(V_R^+)$ and $\Irr(L_{\g}(k_R,0))$}
Let $R$ be an indecomposable root lattice such that $R\not\cong A_1$.
In this appendix, we describe in Tables \ref{T:A2} to
\ref{T:E8} the correspondences between $\Irr(V_R^+)$ and $\Irr(L_\g(k_R,0))$ via the isomorphism $V_R^+\cong L_\g(k_R,0)$ in Proposition \ref{P:VOAgen1} (see also Table \ref{T:gQ}), where $\g=(V_R^+)_1$.
For the notations of irreducible $V_R^+$-modules, see \eqref{Eq:Not}.
For a dominant integral weight $\Lambda$ of a simple Lie algebra $\s$ of level $k$, we denote by $[\Lambda]$ the irreducible $L_\s(k_R,0)$-module $L_\s(k_R,\Lambda)$.
Here we adopt the labeling in \cite[Section 11.4]{Hu} of simple roots and the associated fundamental weights for both $R$ and $\s$.
In these tables, $\lambda_i$ and $\Lambda_i$ mean the fundamental weights for $R$ and for $\s$, respectively.
We omit the detail of central characters $\chi$ for irreducible $V_R^+$-modules of twisted type.

\begin{footnotesize}

\begin{table}[ht]
\caption{Case $R=A_2$} \label{T:A2}
\begin{tabular}{|c|c|}
\hline
$\Irr(V_{A_2}^+)$&$\Irr(L_{A_1}(4,0))$\\ \hline
$(0)^\pm$&$[0]$, $[4\Lambda_1]$\\\hline
$(\lambda_1)$&$[2\Lambda_1]$\\\hline
$(\chi)^\pm$&$[\Lambda_1]$, $[3\Lambda_1]$\\\hline
\end{tabular}
\end{table}

\begin{table}[ht]
\caption{Case $R=A_{2n}$ $(n\ge2)$} \label{T:A2n}
\begin{tabular}{|c|c|}
\hline
$\Irr(V_{A_{2n}}^+)$&$\Irr(L_{B_{n}}(2,0))$\\ \hline
$(0)^\pm$&$[0]$, $[2\Lambda_1]$\\\hline
$(\lambda_i)$ $(1\le i\le n-1)$ &$[\Lambda_i]$ $(1\le i\le n-1)$\\\hline
$(\lambda_n)$&$[2\Lambda_n]$\\\hline
$(\chi)^\pm$&$[\Lambda_n]$, $[\Lambda_1+\Lambda_n]$\\\hline
\end{tabular}
\end{table}

\begin{table}[ht]
\caption{Case $R=A_{3}$} \label{T:A3}
\begin{tabular}{|c|c|}
\hline
$\Irr(V_{A_{3}}^+)$&$\Irr(L_{A_{1}}(2,0)^{\otimes2})$\\ \hline
$(0)^\pm$&$[0]\otimes[0]$, $[2\Lambda_1]\otimes[2\Lambda_1]$\\\hline
$(\lambda_{2})^\pm$&$[2\Lambda_{1}]\otimes[0]$, $[0]\otimes [2\Lambda_1]$\\\hline
$(\lambda_1)$ &$[\Lambda_1]\otimes[\Lambda_1]$\\\hline
$(\chi_i)^\pm$ $(i=1,2)$&$[\Lambda_{1}]\otimes[c\Lambda_1]$, $(c\Lambda_1)\otimes[\Lambda_1]$ $(c\in\{0,2\})$\\\hline
\end{tabular}
\end{table}

\begin{table}[ht]
\caption{Case $R=A_{5}$} \label{T:A5}
\begin{tabular}{|c|c|}
\hline
$\Irr(V_{A_{5}}^+)$&$\Irr(L_{A_{3}}(2,0))$\\ \hline
$(0)^\pm$&$[0]$, $[2\Lambda_2]$\\\hline
$(\lambda_{3})^\pm$&$[2\Lambda_{1}]$, $[2\Lambda_3]$\\\hline
$(\lambda_1)$ &$[\Lambda_2]$\\\hline
$(\lambda_{2})$&$[\Lambda_{1}+\Lambda_3]$\\\hline
$(\chi_i)^\pm$ $(i\in\{1,2\})$&$[\Lambda_{1}]$, $[\Lambda_3]$, $[\Lambda_1+\Lambda_{2}]$, $[\Lambda_2+\Lambda_3]$\\\hline
\end{tabular}
\end{table}

\begin{table}[ht]
\caption{Case $R=A_{2n-1}$ $(n\ge4)$} \label{T:A2n-1}
\begin{tabular}{|c|c|}
\hline
$\Irr(V_{A_{2n-1}}^+)$&$\Irr(L_{D_{n}}(2,0))$\\ \hline
$(0)^\pm$&$[0]$, $[2\Lambda_1]$\\\hline
$(\lambda_{n})^\pm$&$[2\Lambda_{n-1}]$, $[2\Lambda_n]$\\\hline
$(\lambda_i)$ $(1\le i\le n-2)$ &$[\Lambda_i]$ $(1\le i\le n-2)$\\\hline
$(\lambda_{n-1})$&$[\Lambda_{n-1}+\Lambda_n]$\\\hline
$(\chi_i)^\pm$ $(i\in\{1,2\})$&$[\Lambda_{n-1}]$, $[\Lambda_n]$, $[\Lambda_1+\Lambda_{n-1}]$, $[\Lambda_1+\Lambda_n]$\\\hline
\end{tabular}
\end{table}

\begin{table}[ht]
\caption{Case $R=D_{4}$} \label{T:D4}
\begin{tabular}{|c|c|}
\hline
$\Irr(V_{D_{4}}^+)$&$\Irr(L_{A_{1}}(1,0)^{\otimes4})$\\ \hline
$(0)^\pm$&$[0]^{\otimes4}$, $[\Lambda_1]^{\otimes4}$\\\hline
$(\lambda_{i})^\pm$&$\sigma([\Lambda_{1}]^{\otimes2}\otimes[0]^{\otimes2})$, ($\sigma\in\Sym_4$)\\\hline
$(\chi_i)^\pm$ $(i\in\{1,2,3,4\})$&$\sigma([\Lambda_{1}]\otimes[0]^{\otimes3})$,\ $\sigma([\Lambda_1]^{\otimes3}\otimes[0])$, $(\sigma\in\Sym_4)$\\\hline
\end{tabular}
\end{table}

\begin{table}[ht]
\caption{Case $R=D_{6}$} \label{T:D6}
\begin{tabular}{|c|c|}
\hline
$\Irr(V_{D_{6}}^+)$&$\Irr(L_{A_{3}}(1,0)^{\otimes2})$\\ \hline
$(0)^\pm$&$[0]\otimes[0]$, $[\Lambda_2]\otimes[\Lambda_2]$\\\hline
$(\lambda_{1})^\pm$&$[\Lambda_{2}]\otimes[0]$, $[0]\otimes[\Lambda_2]$\\\hline
$\{(\lambda_{5})^\pm\},\ \{(\lambda_{6})^\pm\}$ &$\{[\Lambda_{i}]\otimes[\Lambda_i]\mid i\in\{1,3\}\}$, $\{[\Lambda_{i}]\otimes[\Lambda_j]\mid \{i,j\}=\{1,3\}\}$\\\hline
$(\chi_i)^\pm$ $(i\in\{1,2,3,4\})$&$[c\Lambda_{2}]\otimes[\Lambda_i]$,\ $[\Lambda_i]\otimes[c\Lambda_2]$, $(i\in\{1,3\}$, $c\in\{0,1\}$)\\\hline
\end{tabular}
\end{table}

\begin{table}[ht]
\caption{Case $R=D_{2n}$ $(n\ge4)$} \label{T:D2n}
\begin{tabular}{|c|c|}
\hline
$\Irr(V_{D_{2n}}^+)$&$\Irr(L_{D_{n}}(1,0)^{\otimes2})$\\ \hline
$(0)^\pm$&$[0]\otimes[0]$, $[\Lambda_1]\otimes[\Lambda_1]$\\\hline
$(\lambda_{1})^\pm$&$[\Lambda_{1}]\otimes[0]$, $[0]\otimes[\Lambda_1]$\\\hline
$\{(\lambda_{2n-1})^\pm\},\ \{(\lambda_{2n})^\pm\}$ &
$\{[\Lambda_{i}]\otimes[\Lambda_i]\mid i\in\{n-1,n\}\}$, $\{[\Lambda_{i}]\otimes[\Lambda_j]\mid \{i,j\}=\{n-1,n\}\}$\\\hline
$(\chi_i)^\pm$ $(i\in\{1,2,3,4\})$&$[c\Lambda_{1}]\otimes[\Lambda_i]$,\ $[\Lambda_i]\otimes[c\Lambda_1]$, $(i\in\{n-1,n\}$, $c\in\{0,1\}$)\\\hline
\end{tabular}
\end{table}

\begin{table}[bht]
\caption{Case $R=D_{2n+1}$ $(n\ge2)$} \label{T:D2n+1}
\begin{tabular}{|c|c|}
\hline
$\Irr(V_{D_{2n+1}}^+)$&$\Irr(L_{B_{n}}(1,0)^{\otimes2})$\\ \hline
$(0)^\pm$&$[0]\otimes[0]$, $[\Lambda_1]\otimes[\Lambda_1]$\\\hline
$(\lambda_{1})^\pm$&$[\Lambda_{1}]\otimes[0]$, $[0]\otimes[\Lambda_1]$\\\hline
$(\lambda_{2n-1})$ &$[\Lambda_{n}]\otimes[\Lambda_n]$\\\hline
$(\chi)^\pm$&$[\Lambda_{1}]\otimes[\Lambda_n]$,\ $[\Lambda_n]\otimes[\Lambda_1]$\\\hline
\end{tabular}
\end{table}

\begin{table}[bht]
\caption{Case $R=E_{6}$} \label{T:E6}
\begin{tabular}{|c|c|}
\hline
$\Irr(V_{E_{6}}^+)$&$\Irr(L_{C_{4}}(1,0))$\\ \hline
$(0)^\pm$&$[0]$, $[\Lambda_4]$\\\hline
$(\lambda_{1})$&$[\Lambda_{2}]$\\\hline
$(\chi)^\pm$&$[\Lambda_{1}]$, $[\Lambda_3]$\\\hline
\end{tabular}
\end{table}

\begin{table}[bht]
\caption{Case $R=E_{7}$} \label{T:E7}
\begin{tabular}{|c|c|}
\hline
$\Irr(V_{E_{7}}^+)$&$\Irr(L_{A_7}(1,0))$\\ \hline
$(0)^\pm$&$[0]$, $[\Lambda_4]$\\\hline
$(\lambda_{1})^\pm$&$[\Lambda_{2}]$, $[\Lambda_6]$\\\hline
$(\chi_i)^\pm$ $(i\in\{1,2\})$ &$[\Lambda_{j}]$, $(j\in\{1,3,5,7\})$\\\hline
\end{tabular}
\end{table}

\begin{table}[bht]
\caption{Case $R=E_{8}$} \label{T:E8}
\begin{tabular}{|c|c|}
\hline
$\Irr(V_{E_{8}}^+)$&$\Irr(L_{D_8}(1,0))$\\ \hline
$(0)^\pm$&$[0]$, $[\Lambda_{7}]$\\\hline
$(\chi)^\pm$ &$[\Lambda_{1}]$,\ $[\Lambda_8]$\\\hline
\end{tabular}
\end{table}

\end{footnotesize}

\end{document}